 \documentclass[12pt,oneside]{amsart}
\usepackage[parfill]{parskip}
\usepackage{lscape}
\usepackage{tikz-cd}
\usepackage{amsmath,amsthm,amssymb,amscd}
\usepackage{geometry}
\usepackage{latexsym}
\usepackage{bm}
\usepackage[font=small,labelfont=bf]{caption}
\usepackage{graphicx,epsfig,epstopdf}
\usepackage{placeins}
\usepackage{mathtools}
\usepackage{stackengine}
\usepackage{appendix}
\usepackage{multicol}
\usepackage{enumitem,hyperref}
\usepackage{fourier} 
\usepackage{array,makecell,float}
\usepackage{cleveref}
\usepackage[caption=false]{subfig}
\restylefloat{table}

\stackMath

\providecommand{\R}{\mathbb{R}}

\providecommand{\Z}{\mathbb{Z}}
\DeclareMathOperator{\vol}{vol}

\newcommand{\paren}[1]{\left( #1 \right)}

\newcommand{\abs}[1]{\left\vert#1\right\vert}

\newcommand\blfootnote[1]{%
  \begingroup
  \renewcommand\thefootnote{}\footnote{#1}%
  \addtocounter{footnote}{-1}%
  \endgroup
}
\newcommand{\epsi}{\varepsilon}

\newtheorem{theorem}{Theorem}
\newtheorem{lemma}[theorem]{Lemma}
\newtheorem{conj}[theorem]{Conjecture}

\DeclareMathOperator{\rank}{rank}

\DeclareMathOperator{\CW}{CW}

\newcommand{\PH}{PH}

\begin{document}
\title{Topology and local geometry of the Eden model}


\author{Fedor Manin}
\address{\parbox{\linewidth}{Department of Mathematics\\ University of California, Santa Barbara\\ Santa Barbara, California, USA}}
\email{manin@math.ucsb.edu}
\thanks{First author supported in part by NSF DMS-2001042.}

\author{\'Erika Rold\'an}
\address{\parbox{\linewidth}{Fakult\"at f\"ur Mathematik\\ 
Technische Universit\"at M\"unchen\\ 
Garching b. M\"unchen, Germany}}
\email{erika.roldan@ma.tum.de}
\urladdr{\url{http://erikaroldan.net}}
\thanks{Second author supported in part by NSF-DMS \#1352386 and NSF-DMS \#1812028. This project received funding from the European Union's Horizon 2020 research and innovation program under the Marie Sk\l odowska-Curie grant agreement No.~754462.}

\author{Benjamin Schweinhart}
\address{\parbox{\linewidth}{Department of Mathematics \& Statistics \\ University at Albany \\ Albany, New York, USA}}
\email{bschweinhart@albany.edu}

\keywords{Eden model, first-passage percolation, stochastic topology, topological and geometric data analysis, polyominoes}
\blfootnote{2020 \textit{Mathematics Subject Classification.} 82B43, 62R40, 60K35, 55N31, 05B50.}

\begin{abstract}
   The Eden cell growth model is a simple discrete stochastic process which produces a ``blob'' in $\R^d$: start with one cube in the regular grid, and at each time step add a neighboring cube uniformly at random.  This process has been used as a model for the growth of aggregations, tumors, and bacterial colonies and the healing of wounds, among other natural processes. Here, we study the topology and local geometry of the resulting structure, establishing asymptotic bounds for Betti numbers. Our main result is that the Betti numbers grow at a rate between the conjectured rate of growth of the site perimeter and the actual rate of growth of the site perimeter.  We also present the results of computational experiments on finer aspects of the geometry and topology, such as persistent homology and the distribution of shapes of holes.
\end{abstract}
\maketitle

\section{Introduction}

In this paper, we apply the viewpoint of stochastic topology and topological and geometric data analysis to a discrete geometric model from probability theory: the $d$-dimensional Eden cell growth model (EGM). The 2-dimensional EGM was first introduced and simulated by Murray Eden \cite{eden1958probabilistic,eden1961two} as a model for the growth of colonies of non-motile bacteria on flat surfaces \cite{eden1998history}. It is defined on $\mathbb{R}^2$, using the regular square tessellation of the plane, as follows. Start at time one with one square tile at the origin. At each time step, add a new square tile selected uniformly from among all tiles adjacent to the structure but not yet contained in it (this set of tiles is called the \emph{site perimeter}). This process produces a shape that is well-approximated by a convex set but has interesting geometry at the boundary---see Figure \ref{fig:big}. Here we study the natural higher-dimensional generalization of the EGM to the regular cubical lattice in $\R^d$.

In the probability literature, the EGM is studied as an example of first-passage percolation \cite[Ch.~6]{auffinger201750}, a process which models the spread of a fluid or an infection in a nonhomogeneous medium.  This literature mainly focuses on the large-scale structure and statistics of this process, about which a fair amount is known; one of the most important results is the Cox--Durrett shape theorem \cite{cox1981some}, which shows that under mild assumptions, growth is generally ball-like, rather than fractal as one might initially expect.  That is, over time, the shape of the resulting structure looks more and more like a rescaling of a certain convex set which depends only on the model parameters---so, in the case of the Eden model, only on the dimension.\footnote{While this convex set looks round in simulations in dimension $d=2$, it is known not to be a Euclidean ball for $d>34$~\cite{kesten1986aspects,couronne2011construction}.}

\begin{figure}
    \centering
    
    \subfloat[]{\label{fig:Eden100K}\includegraphics[width=0.33\textwidth]{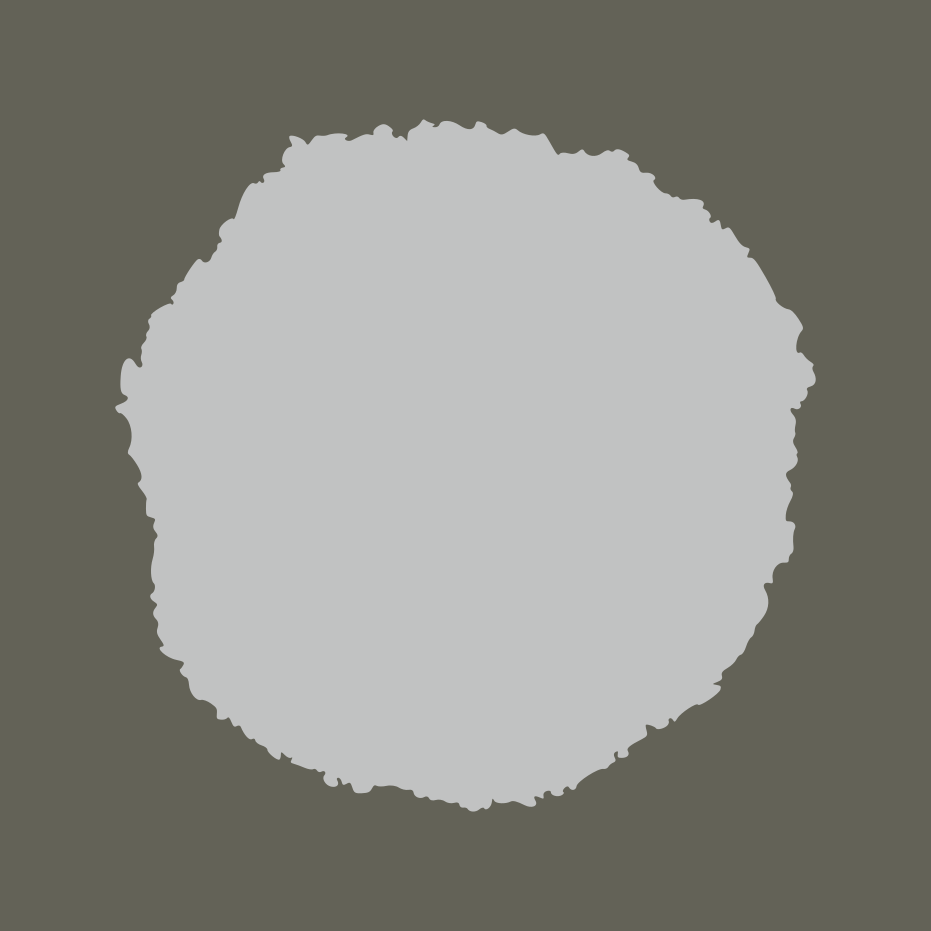}}
    \hspace{.1\textwidth}
    \subfloat[]{\label{fig:lichen}\includegraphics[width=0.33\textwidth]{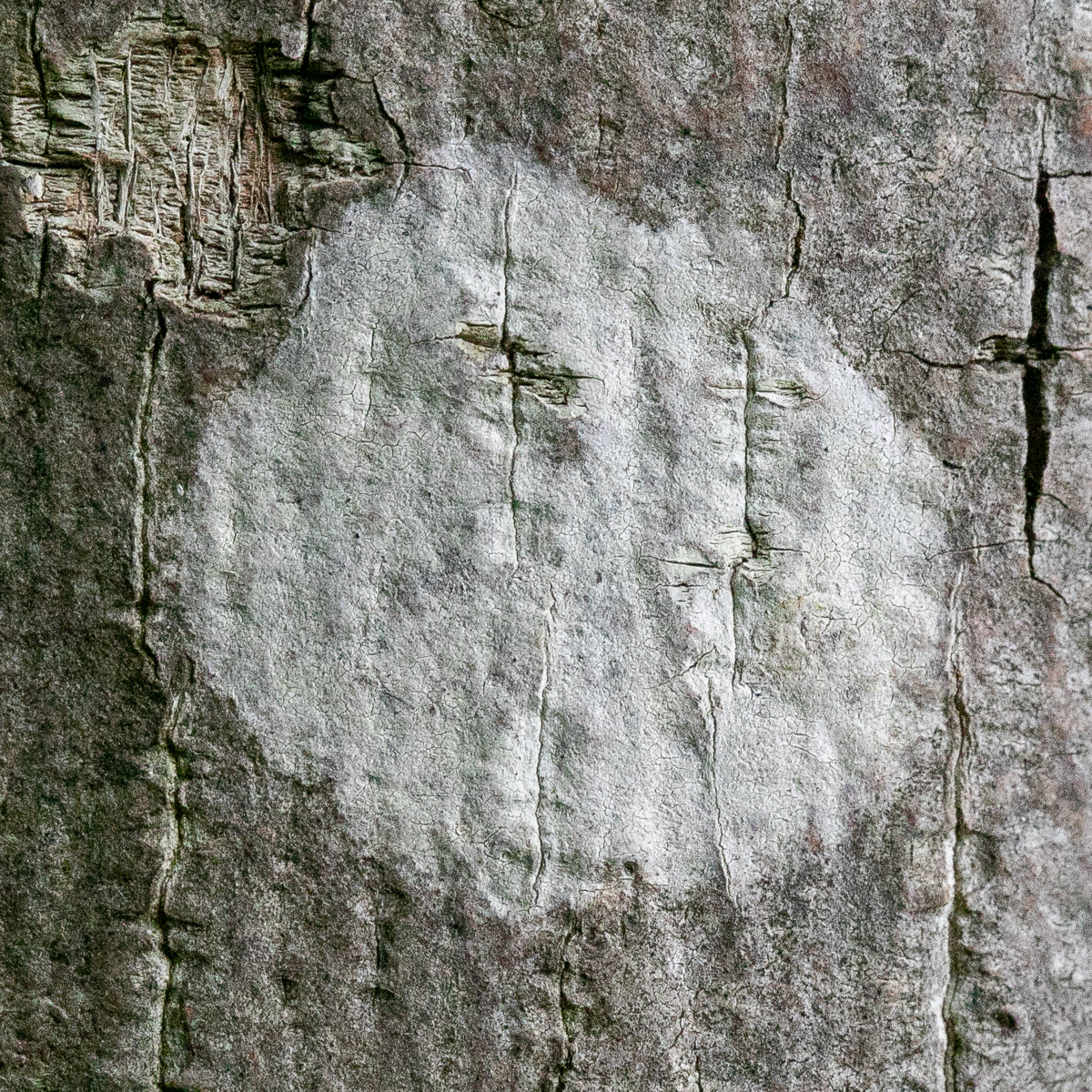}}
    
    \caption{(a) A simulation of the two-dimenional Eden model with 100,000 tiles shown in light gray and (b) the lichen \textit{Phlyctis argena}. Many natural processes result in ``Eden-like'' growth.
    }
    \label{fig:EdenLichen}

\end{figure}

The shape theorem restricts all ``random'' behavior to a collar near the boundary of this convex set which is vanishingly small compared to the whole structure, but whose thickness measured in tiles tends to infinity. As far as we know, not much previous attention has been dedicated to the local geometry in this region for any first-passage percolation model.  This local geometry naturally includes the topology: although holes of any arbitrarily large size eventually appear with probability 1 at the boundary of the Eden model, those holes and other nontrivial cycles get smaller and smaller in comparison to the overall shape with high probability.  Moreover, most of the nontrivial topology occurs at the smallest scales---that is, most of the homology is generated by very small cycles.  In short, by exploring the topology of the Eden model, we quantify small-scale perturbations of the boundary.
 \begin{figure}
    \centering
    \includegraphics[scale=.24]{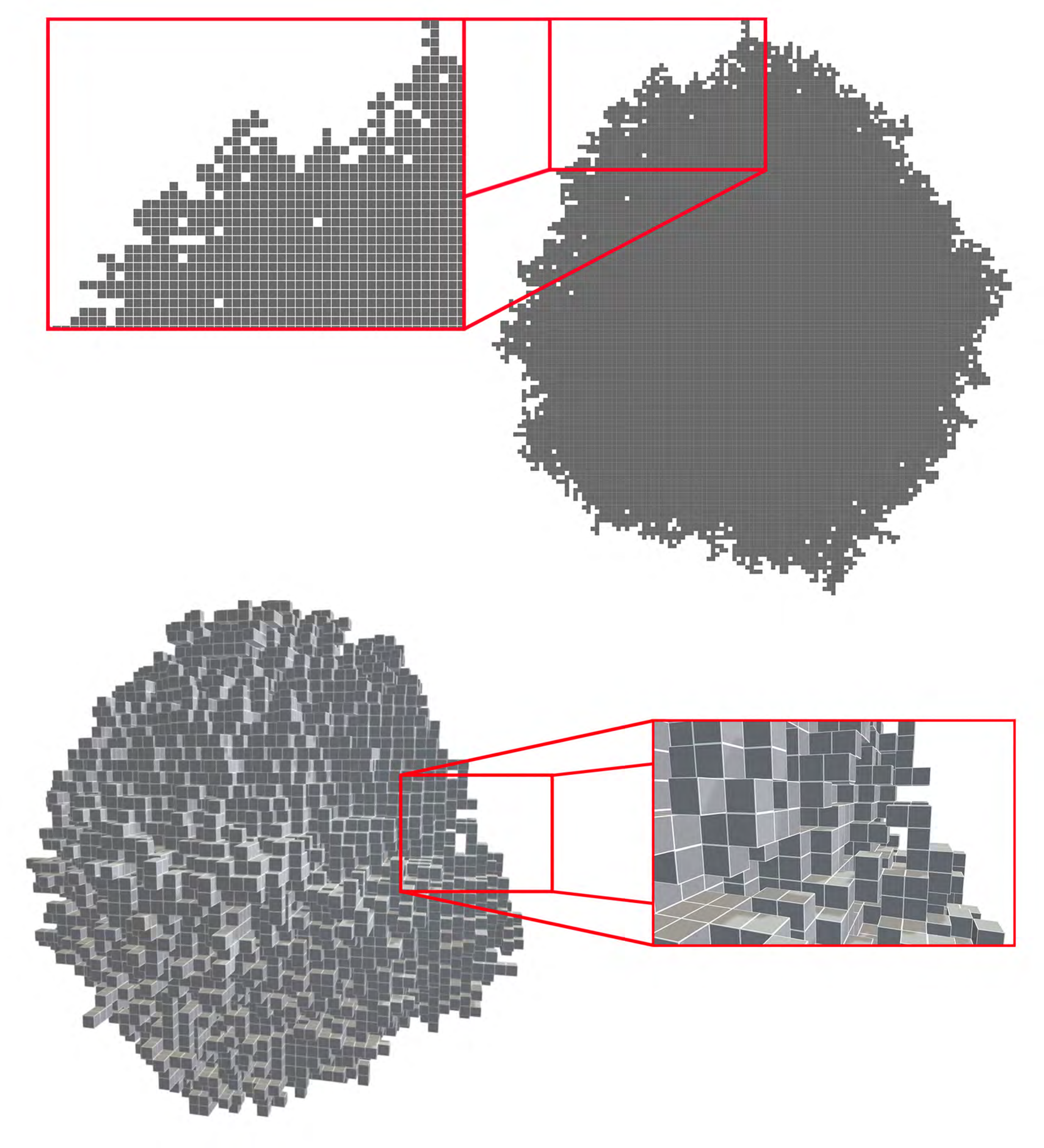}
    \caption{Simulation of the 2D and 3D Eden growth models up to time 10,000 and 30,000 respectively. We zoom in to a portion of the boundary of each model to showcase some local topology generated by one-dimensional holes. The 3D example is available for online interactive exploration at the webpage \url{https://skfb.ly/6SnT9}.}
    \label{fig:big}
 \end{figure}

Stochastic growth models have been applied to study the temporal and spatial dynamics of a wide range of processes including the growth of bacterial cell colonies \cite{vicsek1990self} and tumors \cite{waclaw2015spatial} in biology, the spread of diseases in epidemiology \cite{rhodes1997epidemic}, gelation and crystallization in materials science and physics \cite{family1984kinetics}, and urban growth \cite{teknomo2005unconstrained} in the social sciences. The Eden model is an example of such a model that is simple enough to study analytically yet complex enough to capture important scaling behavior. Its surface is a prototypical model for the growth of interfaces and rough surfaces \cite{barabasi1995fractal,halpin1995kinetic,meakin1998fractals}. For a modified version on a flat substrate, this surface is believed to fall into the Kardar--Parisi--Zhang universality class \cite{kardar1986dynamic, corwin2012kardar, halpin1995kinetic}, a large class of random interfaces characterized by the scaling behavior of the height function. Other systems believed to fall into this class include ballistic deposition and anisotropy-corrected versions of the Eden model in $\Z^d$ \cite{alves2011universal}, while statistics consistent with it have been observed in experiments of paper wetting \cite{PhysRevE.54.685} and turbulent liquid crystals \cite{PhysRevLett.104.230601}. 


The EGM itself has specifically been proposed as a model for wound regeneration \cite{agyingi2018eden} and the growth of bacterial colonies \cite{vicsek1990self}. Similar systems with additional parameters or modified rules for the addition (or subtraction) of tiles have been proposed to model a wide variety of phenomena, such as the magnetic Eden model for aggregations of particles with a fixed spin in a medium\footnote{Constrast with the Ising model, in which the spins of the particles are allowed to change over time.} \cite{ausloos1993magnetic, candia2008magnetic, klymko2017similarity}; cellular automata \cite{deutsch2017growth}; tumor growth \cite{waclaw2015spatial}; and urban growth \cite{ koenig2018system, teknomo2005unconstrained}, among others.

Our results show, informally, that the amount of topology in the Eden model scales roughly with the perimeter. This topology provides a method to characterize the behavior below the interface and distinguishes it from other models which produce similar interfaces on the large scale, but may be topologically trivial or (as in the case of ballistic deposition) have topology which scales with the volume.  Computing the Betti numbers may therefore give additional evidence for or against certain mechanisms of growth. To give a toy example, suppose we locate an interface between populations of two competing species A and B, and we want to understand, based on the synchronic picture, the extent to which one is outcompeting the other.  If species A is not growing its range, we would expect species B to occupy a connected region.  On the other hand, our results indicate that Eden-type growth of species A will \emph{reliably} leave ``voids'', that is, disconnected regions where species B predominates.  If both are growing, then the number of voids on each side will be roughly proportional to the relative growth rate.  Thus the topology of the interface leaves clues about the process of its formation.

Such techniques can perhaps be applied to more complex models found in the literature.  In \cite{kuhr2011range}, an Eden model with mutations is used to simulate the growth of two populations: a ``wild-type'' population and a mutant form spreading within it. The holes in the wild-type and mutant populations exhibit qualitatively different behavior---in both their frequencies and shapes---which depends on the parameters controlling the spread of the mutation (see Figure 3 of \cite{kuhr2011range}).  Topological information could be used to determine such parameters from data.

Furthermore, in applications of stochastic growth models, the geometry of the perimeter strongly influences the interaction with the ambient environment. For example, in materials science, the roughness and porosity are important in a wide variety of contexts, and in marine biology the shape of a coral colony is related to resource acquisition \cite{lartaud2017growth}. This interaction might depend on the local topology of the structure. For example, for a three-dimensional aggregation, the two-dimensional homology corresponds to voids that do not connect to the outside; cells on the surface of these voids do not have the same access to external resources. The one-dimensional homology concentrated on the surface of an aggregation provides a measure of the complexity of that surface; the cells forming a $1$-dimensional homology class could be thought of as a filter in the sense that the medium can flow through them.

\section{Main results}\label{S:main_results}
Our main results concern the rate of growth of the $i$-dimensional homology groups of the Eden growth model. Let $A(t)$ be the $d$-dimensional Eden model at time $t,$ for $d \geq 2,$ and let $\beta_i(t)$ denote the rank of the $i$-dimensional homology (the $i$th \emph{Betti number}) of $A(t)$. Roughly speaking, $\beta_i(t)$ measures the number of ``$i$-dimensional holes'' in $A(t).$ For example, if $d=3$, $\beta_1(t)$ gives the number of tubes through $A(t)$ (for a solid donut, this is one) and $\beta_2(t)$ gives the number of voids of $A(t)$, or bounded components of the complement of $A(t)$ (for a sphere, this is one). See Section \ref{sec:homology} for a technical definition.

The first result relates the growth of $\beta_i(t)$ with that of the \emph{site perimeter} of $A(t)$, the set of tiles adjacent to but not contained in $A(t)$. Write $P_d(t)$ for the volume of the site perimeter.
\begin{theorem}\label{thm_conditional}
    For each $d$ and $1 \leq i \leq d-1$, there is a constant $c=c(d,i)>0$ such that
    \begin{align}
        ct^{(d-1)/d} &\leq \beta_i(t) \leq 2^{d-i}{d \choose i}P_d(t), & i &\leq d-2, \\
        cP_d(t) &\leq \beta_i(t) \leq P_d(t), & i &= d-1,
    \end{align}
    with high probability as $t\rightarrow\infty$.
\end{theorem}
In particular, the rank of the top-dimensional homology (the number of ``voids'') scales with the volume of the perimeter.

Heuristics used in the physics literature \cite{leyvraz1985active} suggest that the volume $P_d(t)$ of the site perimeter of the $d$-dimensional EGM scales as $t^{(d-1)/d}$. This has been proven to be true ``on average'' and ``most of the time'' by Damron, Hanson, and Lam \cite{damron2018size}, but the stronger conjecture that it is true with high probability is wide open. Assuming this conjecture, our theorem shows that the ranks of all homology groups scale with the volume of the perimeter, up to a linear factor. This makes sense on an intuitive level, as any connected local configuration, including those that create topology locally, should occur with some non-zero probability anywhere on the boundary.

The lower bound of Theorem \ref{thm_conditional} is a corollary of a more general result (Theorem \ref{outermost} below): every local configuration (in a cube of sidelength $R$) of filled and empty tiles occurs, with high probability, at least $c(R,d)t^{(d-1)/d}$ times at the boundary of the time-$t$ polyomino.  Thus, for example, cycles of arbitrarily large size, while they are rarer the bigger they are, still occur arbitrarily many times as $t$ increases.


The results of our computational experiments (Section~\ref{sec:comp}) suggest a stronger conjecture about the growth rate:
\begin{conj}\label{edenexp}
    There exists a $C_{i,d}>0$ so that
    \begin{equation} \label{CBettibounds}
        \beta_{i}(t)/t^{\frac{d-1}{d}}\rightarrow C_{i,d}
    \end{equation}
    almost surely as $t\rightarrow\infty$. 
\end{conj}
The constants suggested by our experiments are $C_{1,2}\approx 1.1$ and $C_{1,3}\approx 0.419$.  While we conducted experiments for higher-dimensional homology and higher-dimensional Eden models, we do not have sufficient evidence to provide reasonable guesses for the other constants.

We have also investigated how the rank of the homology can change in one step, proving another theorem:
\begin{theorem}\label{allChanges}
    If $\beta_i(t)$ is the $i$th Betti number of the $d$-dimensional EGM stochastic process at time $t$, then for all $t$
    \begin{equation} \label{eqn:allChanges}
        -2^{d-1-i}{d-1 \choose i} \leq \beta_i(t) - \beta_i(t-1) \leq 2^{d-i}{d-1 \choose i-1},
    \end{equation}
    and all the values, including the extremal values, are attained with positive probability for all $t \geq 3 \cdot 5^{d-1}$.  Moreover, with high probability, each value is attained $\geq ct$ times before time $t$, for some $c=c(d)>0$.
\end{theorem}
Assuming the conjecture on the growth of the perimeter stated above, we can improve the probabilistic portion of this result:
\begin{theorem}\label{condChanges}
    Assume that there is a $C(d)>0$ so that $P_d(t)\leq C(d) t^{(d-1)/d}$ with high probability. Then, for $t>>0$ and for each $-2^{d-1-i}{d-1 \choose i} \leq \ell \leq 2^{d-i}{d-1 \choose i-1}$,
    \[\mathbb{P}\paren{\beta_i(t) - \beta_i(t-1)=\ell} \geq c(d)\]
    for some constant $c(d)>0$.
\end{theorem} 
The constant $c(d)$ in both these results is $\sim 1/\exp(\exp(d))$, and we believe that this is close to optimal for the rarest cases.  Thus even in the 4-dimensional EGM, one cannot expect every possibility to show up in the course of a reasonable-length simulation, as we indeed see in our computational experiments. Proofs of Theorems \ref{allChanges} and \ref{condChanges} are included in Section~\ref{sec:otherproof}. 

Again, our computational experiments suggest stronger regularity properties for the distribution of these jumps:
\begin{conj}
\label{conj:bettiJumps}
For every $-2^{d-1-i}{d-1 \choose i} \leq \ell \leq 2^{d-i}{d-1 \choose i-1}$,
\[\mathbb{P}\paren{\beta_i(t) - \beta_i(t-1)=\ell}\]
converges to a positive constant as $t \to \infty$.
\end{conj}

In Section \ref{sec:comp}, we present the results of our computational experiments for the Eden model. First, we consider the rates of growth of the perimeter (Section~\ref{S:Comp_per}) and the Betti numbers (Section~\ref{S:Comp_betti}), and compare the behavior of $\beta_i(t)$ for different values of $i.$ Next, we apply persistent homology in Section~\ref{S:Comp_persistent} to study the amount of time between when an  $i$-dimensional hole first appears in the Eden model and when it is killed by the addition of tiles. Finally, in Section~\ref{S:Compu_areasandshapes} we consider the distributions of the volumes and shapes of the $d-1$-dimensional holes in the Eden model, and how these holes divide as time progresses. The software and data developed in the course of this research is publicly available on GitHub \cite{EdenSoftware}.

\section{Definitions and preliminaries}

To formally define the Eden model and its homology, we think of the regular cubic tiling as endowing $\R^d$ with the structure of an infinite cubical complex whose vertices are $\Z^d$ and whose $d$-cells are translates of $[0,1]^d$.  We call this cubical complex $\CW(\Z^d)$.  A \emph{polyomino} is a union of $d$-cells of this structure (a \emph{pure $d$-dimensional subcomplex}) which is \emph{strongly connected}, that is, its interior is connected; in other words, the interiors of any two $d$-cells are connected via a path which is disjoint from the $(d-2)$-skeleton (cf.\ the definition of a pseudomanifold).  In the combinatorics literature, these are known as \emph{polyominoes} in 2 dimensions and \emph{polycubes} in 3 dimensions.

Given a polyomino $A$, its $i$-skeleton $A^i$ is the union of all $i$-cells in $A$, forming a filtration
\[A^0 \subset A^1 \subset \cdots \subset A^d=A.\]

The \emph{site perimeter} of a polyomino $A$ is the set of $d$-cells of $\CW(\Z^d)$ that are not in $A$ but have $(d-1)$-cells in common with $A$; in other words, $d$-cells $Q$ such that $A \cup Q$ is again a polyomino.  This contrasts with the \emph{boundary} of the polyomino, which is a $(d-1)$-dimensional complex defined using the usual topological notion $\partial A=\overline{A} \cap \overline{A^c}$.\begin{figure} 
\begin{center}
\begin{tikzpicture}[scale = .7] 

\foreach \x in {-1,0,1,2,3} {
    \foreach \y in {-1,0,1,2,3} {
    \path [draw=gray!33!white, fill=gray!33!white] (\x+0.05, \y+0.05)
    -- ++(0,.9) -- ++(.9,0) -- ++(0,-.9) --cycle;
    \path [draw=gray!33!white, fill=gray!33!white] (\x+7.05, \y+0.05)
    -- ++(0,.9) -- ++(.9,0) -- ++(0,-.9) --cycle;
    }
}

\foreach \x/\y in {0 / 0, 0/1, 0/2, 1/0, 1/2, 2/0, 2/1 , 7 / 0, 7/1, 7/2, 8/0, 8/2, 9/0, 9/1, 9/ 2 } { 
\path [draw=gray!66!black, fill=gray!66!black] (.5+\x-0.45, .5+\y-0.45)
-- ++(0,.9)
-- ++(.9,0)
-- ++(0,-.9)
--cycle;
}
\foreach \x/\y in { .5/-.5, 1.5/-.5, 2.5/-.5, 3.5/.5, 3.5/1.5, 2.5/2.5, -.5/.5, -.5/1.5, -.5/2.5, 1.5/1.5, 1.5/3.5, .5/3.5}{
\draw [fill=black, black] (\x,\y) circle [radius=.1];
}

\foreach \x/\y in { .5/-.5, 1.5/-.5, 2.5/-.5, 3.5/.5, 3.5/1.5, 3.5/2.5, 2.5/3.5, -.5/.5, -.5/1.5, -.5/2.5, 1.5/1.5, 1.5/3.5, .5/3.5}{
\draw [fill=black, black] (\x+7,\y) circle [radius=.1];
}

\end{tikzpicture}
\caption{Two polyominoes with the site perimeter highlighted.  Each has one 1-dimensional hole, i.e. $\beta_1 = 1$.}
\label{figure:one}
\end{center}
\end{figure}
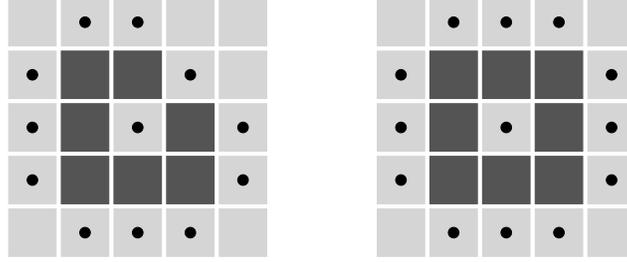

The Eden cell growth model is a stochastic process which produces a polyomino $A(t)$. It starts at time 1 with one $d$-cube at the origin, and at each time step, $A(t+1)=A(t) \cup Q_{t+1}$ where $Q_{t+1}$ is a $d$-cube chosen uniformly at random from the site perimeter.

The Eden model is often equivalently defined with the cubes replaced by vertices of the lattice $\mathbb{Z}^d$, thought of as a graph with neighboring vertices linked along each axial direction.  At each time step, a single unfilled vertex along the site perimeter is filled.  In this formulation, the site perimeter consists of unfilled vertices which share an edge with a filled vertex, and the boundary consists of edges between filled and unfilled vertices (the boundary of the set of filled vertices in the sense of graphs).  Our definition in terms of cubes is needed to define the homology of the Eden model; we take note of this equivalent formulation because it is the usual way of formalizing first-passage percolation, as we describe below.

\subsection{Homology}
\label{sec:homology}
The homology groups of a space are a sequence of abelian groups representing the ``$i$-dimensional holes'' of the complex.  For example, a solid donut has a single 1-dimensional hole, while a 2-sphere has a single 2-dimensional void; these correspond to the ranks of the homology groups $H_1$ and $H_2$, respectively.  A ``0-dimensional hole'' is a disconnection, and the rank of $H_0$ is the number of connected components of the space.  Homology groups of cubical complexes are most easily defined combinatorially, but are topological invariants.  The reader is referred to an algebraic topology textbook such as \cite{hatcher2002algebraic} for more information.

In this paper we use homology with coefficients in the field $\mathbb{F}_2=\{0,1\}$; we suppress this in our notation.  Given a cubical complex $A$, let $C_i(A)$ be the vector space of \emph{$i$-chains}, that is, formal $\mathbb{F}_2$-linear combinations of $i$-cells.  The boundary homomorphism $\partial_i:C_i(A) \to C_{i-1}(A)$ sends each cell to the formal sum of the $(i-1)$-cells on its boundary.  Then the $i$th homology is the vector space of \emph{$i$-cycles}, which have zero boundary, modulo the $i$-dimensional boundaries of $(i+1)$-chains:
\[H_i(A)=\ker(\partial_i)/\partial_{i+1}(C_{i+1}(A)).\]
The $i$th \emph{Betti number} $\beta_i(A)$ is the dimension of $H_i(A)$.  Thus $\beta_0(A)$ is the number of connected components---always 1 for a polyomino.  Moreover, for a $d$-dimensional polyomino, $\beta_i(A)=0$ for all $i \geq d$.  This is obvious for $i>d$ and true for all subsets of $\mathbb{R}^d$ in the case $i=d$.  This leaves the cases $1 \leq i \leq d-1$ as the interesting ones to measure for the Eden model.

\subsection{First-passage percolation and the Eden model} \label{S:FPP}
First-passage percolation (FPP) is a well-studied family of stochastic processes on the lattice $\mathbb{Z}^d$, thought of as a graph; see \cite{auffinger201750} for an extensive survey.  Here we describe how the Eden model can be thought of as a special case of FPP, which will be useful in several of our proofs.

We first define two types of stochastic processes.  In \emph{bond FPP}, the lattice is given by a graph metric with edge lengths pulled i.i.d.\ from some probability distribution, and the process of interest is the growth of the $t$-ball around the origin in this metric.  \emph{Site FPP} is similar but a bit harder to define; here every vertex of the graph (called a \emph{site}) is assigned an i.i.d.\ number called a \emph{passage time}.  The passage time of a site $p$ governs the time from when a site adjacent to $p$ first gets ``infected'' to when $p$ gets infected.  We again start with the origin infected at time $0$ and study the set of infected sites at time $t$.

Now consider site FPP where the passage times are distributed exponentially with mean $1$.  The exponential distribution is important because it is ``memoryless'' in the sense that
\[P(X>t+s \mid X>s)=P(X>t).\]
Thus, conditioning on the event that the ball at time $t$ is a polyomino $A$, the additional time required to add a specific adjacent site is again exponential with mean $1$, and is independent from when other adjacent sites are added and from the passage times of non-adjacent unfilled sites. In particular, every site in the perimeter has the same probability of being infected next.  But this is exactly how the Eden model works, except that in the Eden model the time to add the next tile is fixed. Consequently, the Eden model can be thought of as a (variable) time rescaling of this FPP model.  This was first observed by Richardson \cite{richardson1973random}.

\subsection{Variations on the model}
Our results are stable with respect to certain variations on the setup described above.

First, instead of uniformly selecting a tile in the site perimeter, one could uniformly select an face on the boundary and add the adjacent tile along the face.  In other words, the probability that an element of the site perimeter is selected is weighted by the number of connections between it and the polyomino at time $t$. This can be modeled using first-passage percolation like the usual Eden model, but using bond FPP rather than site FPP.  All of our proofs can easily be modified to produce analogous results for this model.

Another potential variation relates to how the topology of the Eden model is defined; rather than connecting cubes that touch at corners, one could consider two cubes to be connected only if they share a face.  The advantage of this idea is that this aligns with the notion of adjacency used in defining growth.  There are several ways of formalizing this idea. One is to consider the interior of the cubical complex constructed above.  Alternatively, one can build a new cubical tessellation by placing grid points at the centers of cubes of the polyomino; the intersection of this tessellation with our polyomino is a deformation retract of its interior, and this gives a combinatorial characterization.  The proof of Theorem \ref{thm_conditional} works without modification with this redefinition.  One can also get an analogue of Theorems \ref{allChanges} and \ref{condChanges}, though with different constants: to understand the effect of adding a cube one has to work with the geometry of its dual cross polytope, rather than of the cube itself.  In the end, though, this variation simply switches the role of the Eden ball and its complement: Alexander duality tells us that a sufficiently nice domain gives the same topological information as the closure of its complement, and we can obtain such a domain either by slightly thickening the Eden ball or by slightly thickening its complement.

Finally, our results can easily be extended to other regular tessellations of $\R^d$ besides the cubical one.  In fact, much of what we say seems to depend only on the large-scale geometry of the contractible cell complex $\CW(\Z^d)$.  One direction for further research would be to understand similar models on tessellations of hyperbolic space, nilpotent Lie groups, other symmetric spaces, CAT(0) cube complexes, and other contractible spaces on which a group acts geometrically. For what little is known about first-passage percolation on spaces of interest in geometric group theory, see \cite{BenT}.

\subsection{Combinatorics of cubes and polyominoes}
The following is easy to see:
\begin{lemma}
    The number of $i$-dimensional faces of the $d$-dimensional cube is $2^{d-i}{d \choose i}$.
\end{lemma}

In the proof of Theorem \ref{thm_conditional} we require the following combinatorial fact about polyominoes in general.
\begin{lemma} \label{lem:proj}
    Let $A$ be any polyomino in $\mathbb{R}^d$.  Then for some $1 \leq i \leq d$, the projection of $A$ to the $i$th coordinate hyperplane (denoted $\pi_i(A)$) has
    \[\vol_{d-1}(\pi_i(A)) \geq \vol_d(A)^{(d-1)/d}.\]
\end{lemma}
\begin{proof}
    The isoperimetric inequality for polyominoes \cite{BL}, attained by cubes, is
    \[\vol_{d-1}(\partial A) \geq 2d\vol_d(A)^{(d-1)/d}.\]
    Suppose first that $A$ is \emph{convex}, that is the intersection of $A$ with any line parallel to any coordinate axis is connected.  (Note that a convex polyomino is not a convex set!) This is equivalent to saying that every $(d-1)$-cube in $\partial A$ is visible from infinitely far in some coordinate direction.  In that case,
    \[\vol_{d-1}(\partial A)=\sum_{i=1}^d 2\vol_{d-1}(\pi_i(A)),\]
    which completes the proof.
    
    Now take a general polyomino $A$.  We will construct a convex polyomino $A_d$ with the following properties:
    \begin{enumerate}[label=(\roman*)]
    \item $\vol_d(A_d)=\vol_d(A)$.
    \item For each $i$, $\vol_{d-1}(\pi_i(A_d)) \leq \vol_{d-1}(\pi_i(A))$.
    \end{enumerate}
    This comparison proves the lemma for $A$.
    
    We construct $A_d$ by ``lining up'' the columns of $A$ in each coordinate direction.  That is, let $A_0=A$.  Once we have built $A_{i-1}$, we make
    it into $A_i$ by turning on gravity in the $i$th direction and ``shaking'', that is, letting all the cubes fall down to some hyperplane below the polyomino.
    
    Clearly condition (i) holds.  We need to show that (ii) holds and that $A_d$ is column convex.  We show both of these by analyzing each shake, that is, each transition from $A_{j-1}$ to $A_j$.
    
    During the $i$th shake, the polyomino becomes \emph{column convex} in the $i$th coordinate direction, that is, its intersection with any line in that direction is connected.  It remains to show that during subsequent shakes, $j>i$, this convexity is preserved.  Since what happens to a cube depends only on its column, we look at the intersection of the polyomino with each plane in the $ij$-direction. If we start with connected columns lined up on one side, then the $j$th shake sorts those columns by height, without changing their convexity.
    
    Finally, we show that each $j$th shake does not increase the volume of $\pi_i(A)$. Certainly if $i=j$ the projection doesn't change.  Otherwise we again look at the intersection with each plane in the $ij$-direction.  After the shake, what we see from the $i$th coordinate direction is the height of the largest column.  Previously, every cube in that column was either visible or obstructed by something, so the volume of the projection can only decrease.
\end{proof}

\section{Proof of Theorem \ref{thm_conditional}} \label{sec:proof}

We start with the (easy) upper bound.  Write $A(t)$ for the polyomino at time $t$.  Applying the Mayer--Vietoris sequence to $A(t) \cup \overline{A(t)^c}=\mathbb{R}^d$, we see that
\[H_i(\partial A(t)) \cong H_i(A(t)) \oplus H_i(\overline{A(t)^c}).\]
The rank of the left side is bounded by the number of $i$-cells in the boundary, giving the bound $\beta_i(t) \leq 2^{d-i}{d \choose i}P_d(t)$ since $2^{d-i}{d \choose i}$ is the number of $i$-cells in a $d$-cube.

In the case $i=d-1$, we can get a stronger bound since $\beta_{d-1}(t)$ is the number of voids in $A(t)$, in other words, the number of bounded connected components of its complement.  Since every connected component of the complement must include a cell of the site perimeter, $\beta_{d-1}(t) \leq P_d(t)$.

We now prove the lower bound.  Here is the basic outline.  Given a time $t$, we find $\Omega(t^{(d-1)/d})$ disjoint empty boxes of side length $R$ at the perimeter of a somewhat earlier stage $A(t_0)$.  Then we show that once we reach time $t$, at least a constant proportion of these boxes end up containing a structure which adds one to the $i$th Betti number.

The boxes are obtained as follows.  By Lemma \ref{lem:proj}, the projection of $A(t_0)$ in some coordinate direction has volume at least $t_0^{(d-1)/d}$.  Thus (thinking of that direction as ``up'') we can drop $\Omega(t_0^{(d-1)/d})$ boxes from overhead so that they land in different places on top of the polyomino $A(t_0)$.  We formalize this in proving the following more general result.
\begin{theorem}\label{outermost}
    Let $S$ be any $d$-dimensional polyomino which is contained in the cube $[0,R]^d$ and includes the entire base of that cube (i.e.\ $[0,R]^{d-1} \times [0,1]$).  There is a constant $c=c(R,d)>0$ so that $S$ occurs (perhaps in rotated form) as the intersection of $A(t)$ with at least $ct^{(d-1)/d}$ different cubes of width $R$,  with high probability as $t \to \infty$.
\end{theorem}
Before proving Theorem \ref{outermost}, we use it to finish the proof of Theorem \ref{thm_conditional}. Let $1 \leq i \leq d-1$ and set
\[S=([0,5]^{d-1} \times [0,1]) \cup ([2,3]^{d-i-1} \times [1,4]^{i+1}) \setminus [2,3]^d \subset [0,5]^d\,.\]
That is, $S$ is the base together with a ``handle'' homotopy equivalent to $S^i$.  The theorem guarantees $ct^{(d-1)/d}$ copies of $S$ whose intersection with the remainder of $A(t)$ is contained in the base.  Thus $A(t)$ is the union of two pieces: all the copies of $S$ on one side, and the rest of $A(t)$ together with the bases of the copies of $S$ on the other; the intersection is a disjoint union of contractible components, one for each copy of $S$.  Thus by the Mayer--Vietoris theorem, $\beta_i(t) \geq ct^{(d-1)/d}$.

\subsection{Proof of Theorem \ref{outermost}}

To prove Theorem \ref{outermost} we will use the reformulation of the Eden model in terms of first-passage percolation, as described in \S\ref{S:FPP}.  We now keep track of time in the FPP model, which we indicate by $r$ to contrast with $t$ for Eden time and to suggest that it is roughly the radius of the polyomino; the notation $A(r)$ and $P_d(r)$ indicates the Eden model in FPP time and the volume of its site perimeter for the rest of the section.  We also write $|A(r)|$ for the volume of $A(r)$, i.e.\ $t$.  Finally we define the \emph{passage time from $r$} to be the passage time of a site if it is not in the site perimeter of $A(r)$, and the time from $r$ to infection if it is. The memorylessness of the exponential distribution implies that, given $A(r)=A$, the passage times from $r$ to sites not in $A$ are are i.i.d.\ exponential, with no difference between sites in and outside the site perimeter.

Our approach is to find at least $c(d)|A(r-2)|^{(d-1)/d}$ copies of $S$ in $A(r)$ with high probability. Thus, to prove the theorem, we also need to know that $|A(r)| \leq C(d)|A(r-2)|$.  This follows from the Cox--Durrett shape theorem \cite{cox1981some}, which shows in particular that there is a constant $V_0$ such that for every $\epsi>0$, with high probability
\begin{equation*}
(1-\epsi)V_0r^d<|A(r)|<(1+\epsi)V_0r^d.
\end{equation*}
However, in the interest of keeping the overall argument elementary we also provide the following much cruder estimate.  Since this estimate is stated in terms of the site perimeter, it is also useful for our later argument about $\beta_{d-1}(t)$.
\begin{lemma} \label{lem:growth}
    With high probability as $r \to \infty$,
    \[|A(r)| \leq |A(r-2)|+CP_d(r-2) \leq (1+2dC)|A(r-2)|,\]
    where $C=C(d)$ is a constant.  In particular, $P_d(r) \leq (1+2dC)P_d(r-2)$.
\end{lemma}
\begin{proof}
    We will show that there is an $\epsi=\epsi(d)>0$ such that with high probability, $|A(r+\epsi)| \leq |A(r)|+P_d(r)$, and in particular $P_d(r+\epsi) \leq 2P_d(r)$. This will imply the lemma with $C=2^{\lceil 2/\epsi \rceil}$.
    
    Let $\epsi$ be so that $\mathbb{P}(\rho_p<\epsi)=\frac{1}{2d+1}$, where $\rho_p$ is the passage time from $r$ at any site $p \notin A(r)$.  Consider a rooted infinite $(2d-1)$-ary tree equipped with passage times on the nodes distributed via the same exponential distribution.  The expected size $E$ of the maximal subtree containing the root (if nonempty) whose nodes all have passage times $<\epsi$ satisfies the recurrence relation
    \[E=\frac{1}{2d+1}(1+(2d-1)E);\]
    thus $E=\frac{1}{2}$.  This bounds the expected size of the subtree reached in time $\epsi$.
    
    Now we show that $V(r+\epsi)$ is bounded above by the total size of all these subtrees for a collection of $P_d(r)$ independent such trees.  We associate the roots of the trees to the cubes of the site perimeter of $A(r)$, and then map each tree to $\mathbb{Z}^d$ via a graph homomorphism by thinking of paths in the tree as corresponding to reduced words on $d$ letters and their inverses, with one letter missing from the initial position corresponding to a neighbor of the root site which is in $A(r)$.
    
    We give a coupling between the weight distribution on the collection of trees and the passage times from $r$ for sites outside $A(r)$, in which each site is coupled to some node which maps to it.  Namely, the nodes in the site perimeter are associated to the corresponding tree.  Then we couple each subsequent site to a neighbor of the node coupled to the neighboring site reached at the earliest time.  Thus the coupling between the probability spaces depends on the values pulled from preceding distributions; this doesn't affect any probabilities since all that changes is \emph{which} i.i.d.\ exponentially distributed weight corresponds to a given site.
    
    In the end, every site in $A(r+\epsi)$ is coupled to a node which is reached at time $<\epsi$.  Since the expected number of nodes in each tree attained after time $\epsi$ is less than $1$, with high probability, $|A(r+\epsi)|<|A(r)|+P_d(r)$.
\end{proof}

By Lemma \ref{lem:proj}, the projection of $A(r-2)$ in some coordinate direction (without loss of generality, the $x_d$ direction) has volume $\geq |A(r-2)|^{(d-1)/d}$.  In particular, if we partition the plane $x_d=0$ into coordinate cubes of side length $R$, some number $N \geq R^{-(d-1)}|A(r-2)|^{(d-1)/d}$ of those cubes intersect this projection.  For each such cube $K_j$, let $h(K_j)$ be the maximal $x_d$-coordinate of a point of $A(r-2)$ whose $d$th projection lies in $K$.  Thus the $d$-dimensional cube $\tilde K_j=K_j \times [h(K_j),h(K_j)+R]$ touches, but does not intersect $A(r-2)$.

We finish by showing:
\begin{lemma}
    There is a $c(R)>0$ such that with high probability, for at least $c(R)N$ values of $j$, $1 \leq j \leq N$, we have $\tilde K_j \cap A(r)=S+y_j$, where $y_j$ is defined so that $\tilde K_j=[0,R]^d+y_j$.
\end{lemma}
\begin{proof}[Proof of lemma]
%
    For a site $p \notin A(r-2)$, let $\rho_p$ denote its passage time from $r-2$.  As outlined above, the $\rho_p$ are i.i.d.\ for all points outside $A(r-2)$.  Let $X_j$ be the event that for all $p \in \tilde K_j$,
    \begin{align*}
        \rho_p &\leq R^{-d} && \text{if }p \in S+y_j \\
        \rho_p &\geq 3 && \text{if }p \notin S+y_j.
    \end{align*}
    Clearly, the $X_j$ are i.i.d.\ and each $X_j$ occurs with positive probability.  Therefore there is a constant $c(R)>0$ such that with high probability at least $c(R)N$ of the $X_j$ occur.
    
    Now notice that if $X_j$ occurs, then for some $p$ in the base of $\tilde K_j$, $A(r-2+R^{-d})$ contains $p$.  Every point in $S+y_j$ is connected to that point by a path through $S+y_j$ of length certainly $\leq R^d-1$. Therefore, for $r-1<s<r+1$, $A(s)$ contains all the points of $S+y_j$ and none of the points of $\tilde K_j \setminus (S+y_j)$.  This proves the lemma.
\end{proof}

\subsection{Proof of the lower bound for top-dimensional holes}

Finally, we prove the stronger lower bound $\beta_{d-1}(t) \geq cP_d(t)$.  We will show the following:
\begin{lemma}
    There is a $c>0$ such that with high probability, there are at least $cP_d(r-2)$ voids of volume 1 in $A(r)$.
\end{lemma}
Since by Lemma \ref{lem:growth}, $P_d(r) \leq C(d)P_d(r-2)$, this suffices.
\begin{proof}
    For $\sigma \in \mathbb{Z}^d$, let $\Psi_\sigma$ be the set of sites in the intersection of $\sigma+3\mathbb{Z}^d$ with the site perimeter of $A(r-2)$. We can choose $\sigma$ so that $\Psi_\sigma$ contains at least $1/3^d$ of the site perimeter.
    
    Given a site $p \in \Psi_\sigma$, let $X_p$ be the event that the passage time from $r-2$ is $>2$ for $p$ and $\leq 1/2$ for all sites that share a $(d-2)$-face with $p$.  The $X_p$ are i.i.d.\ and each $X_p$ occurs with positive probability.  Therefore there is a constant $c>0$ such that with high probability, at least $c \cdot 3^d|\Psi_\sigma|$ of the $X_p$ occur.
    
    It is easy to see that if $X_p$ occurs, then $A(r)$ contains all the neighbors of $p$, but not $p$.
\end{proof}

\section{Proof of Theorems \ref{allChanges} and \ref{condChanges}}\label{sec:otherproof}

We now endeavor to understand the possible changes in $\beta_i$ at a single timestep.  Let $A(t)$ be the polyomino at time $t$, and $Q$ be the tile added at time $t+1$.  Then by excision and the long exact sequence of a pair, we have
\[H_i(A(t+1),A(t)) \cong H_i(Q,Q \cap A(t)) \cong \tilde H_{i-1}(Q \cap A(t)),\]
where $\tilde H_i$ indicates reduced homology.  The long exact sequence of the pair $(A(t+1),A(t))$ then indicates that
\[-\max \rank H_i(Q \cap A(t)) \leq \beta_i(t+1)-\beta_i(t) \leq \max \rank H_{i-1}(Q \cap A(t)),\]
where the maximum is taken over possible subcomplexes of the $d$-dimensional cube which could be $Q \cap A(t)$.

We now compute this maximal rank.  Notice that $Q \cap A(t)$ has to include at least one $(d-1)$-dimensional face in order for us to be able to add the tile $Q$.  Without loss of generality, we assume this is the base of the cube.  Since adding $i$-cells can only increase it and adding $(i+1)$-cells can only decrease it, $\rank H_i(Q \cap A(t))$ is maximized when $A(t)$ includes the entire $i$-skeleton of the cube but no $(i+1)$-cells outside the base.

Write $Q_r=[0,1]^r \subset \mathbb{R}^d$ equipped with the standard cell structure.  Then it is enough to compute
\[\rank \tilde H_i(Q_{d-1} \cup Q_d^{(i)})\]
where $ Q_d^{(i)}$ is the $i$-skeleton of $Q_d$. Notice that $Q_{d-1} \cup \bigl(Q_{d-1}^{(i)} \times [0,1]\bigr)$ is contractible and obtained by adding $(i+1)$-cells to $Q_{d-1} \cup Q_d^{(i)}$; the number of these $(i+1)$-cells is the same as the number $J(d-1,i)$ of $i$-cells in the $(d-1)$-cube.  Therefore
\[\rank \tilde H_i(Q_{d-1} \cup Q_d^{(i)})=J(d-1,i)=2^{d-1-i} {d-1 \choose i}.\]
This demonstrates equation \eqref{eqn:allChanges}.

It remains to show that every change in $\beta_i$ within this range is attained by some configuration. The Eden model produces any polyomino with positive probability, so it is enough to demonstrate:
\begin{lemma} \label{jump_configs}
\begin{enumerate}[label=(\alph*)]
    \item For each $1 \leq i \leq d-1$ and each $0 \leq k \leq J(d-1,i)$, there is a polyomino  in which adding a tile decreases $\beta_i$ by $k$ and increases $\beta_{i+1}$ by $J(d-1,i)-k$.
    \item For each $1 \leq k \leq J(d-1,0)=2^{d-1}$, there is a configuration in which adding a tile increases $\beta_1$ by $k$.
\end{enumerate}
\end{lemma}
\begin{proof}
    Given subcomplexes $R \subseteq S \subseteq Q_d^{(d-1)}$, we will construct a set $A_{R,S}$ of tiles in the $5 \times \cdots \times 5$ grid centered at $Q_d$ that is homotopy equivalent to $S$ and intersects $Q_d$ in $R.$ A tile $Q'$ adjacent to $Q_d$ is included if and only if $Q' \cap Q_d$ is contained in $R$.  The tiles in the boundary of the $5\times \cdots \times 5$ grid  are included according to the following criterion.  The planes containing the $d-1$-faces of $Q_d$ partition $\mathbb{R}^d$ into $3^d$ regions. The intersection of the closure of such a region with $Q_d$ consists of exactly one face of $Q_d.$  A boundary tile not in the top or bottom layer is included if and only if the region containing it intersects $Q_d$ in a face of $S$.
    
    \begin{figure}
    \centering
        \includegraphics[width=0.4\textwidth, scale=1.8]{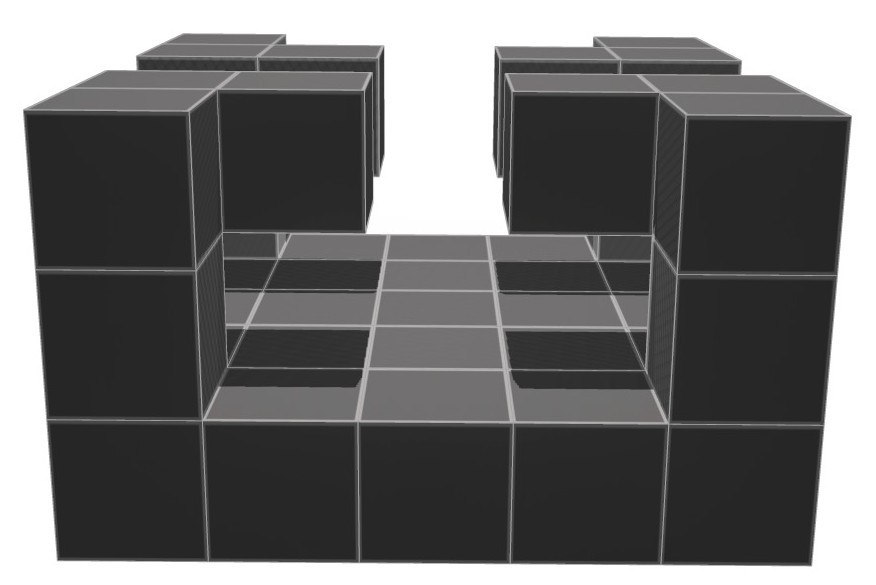}
    \qquad
        \includegraphics[width=0.4\textwidth, scale=1.8]{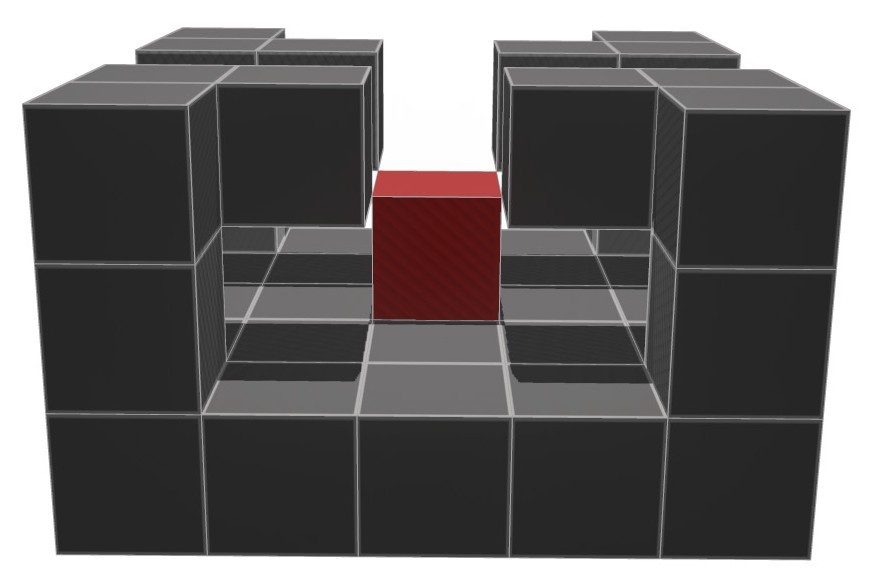}
    \caption{Adding the central cube to this configuration increases $\beta_1$ by $4$; this is the construction given in the proof of Lemma \ref{jump_configs}(b), altered slightly for visibility.}
    \end{figure}
    
    $A_{R,S}$ is homotopy equivalent to $S$, and $A_{R,S} \cap Q_d=R$; thus using the Mayer--Vietoris theorem one sees that
    \[H_i(A_{R,S} \cup Q_d) \cong H_i(S,R).\]
    Then to fulfill (a) we use $A_{R,S}$ with $R=Q_{d-1} \cup Q_d^{(i)}$ and
    $R \subseteq S \subseteq Q_{d-1} \cup Q_{d-1}^{(i)} \times [0,1]$, with $S$ containing $k$ of the extra $(i+1)$-dimensional faces.  To fulfill (b) we use $A_{R,S}$ with $R$ comprising $Q_{d-1}$ and $k$ vertices of the upper face of $Q_d$, and with $S$ adding in the vertical edges connecting those vertices to the base.  In both cases, adding in the center tile changes the topology as desired.
\end{proof}
Now we show that with high probability, each such jump happens at least $ct$ times between time $0$ and $t$ for some $c=c(d)>0$.  In particular, we show that a constant percentage of the time, the tile added at step $s$ is locally configured as in Lemma \ref{jump_configs}, and that local configuration is attached to the rest of the polyomino only by the base; hence by the Mayer--Vietoris theorem, the change in the overall Betti numbers is the same as the change in the local Betti numbers.

For this we use the FPP formulation of the Eden model; in fact our proof works in a wide range of FPP models.
\begin{theorem}
    Consider a site FPP model in $\mathbb{Z}^d$ whose probability distribution on passage times not supported away from $0$ or away from $\infty$.  We denote the polyomino at time $r$ in this model by the random variable $A(r)$.  Then there is some $c(R)>0$ depending on the distribution such that the following holds.  Let $K$ be a (strict) sub-polyomino of $[0,R]^d$ which contains the entire boundary, and mark a tile $x_0$ inside $[0,R]^d$ and outside but adjacent to $K$. We say that $x \in P_K$ if the tile $x$ is added to the polyomino at a time $r_x \leq r$, and
    \[A(r_x) \cap ([0,R]^d+(x-x_0))=K\cup x+(x-x_0).\]
    Then with high probability $|P_K| \geq c(R)|A(r)|$.
\end{theorem}
Applying this to the configurations in Lemma \ref{jump_configs}, framed inside a filled shell with extra white space added so that only the base of the interior configuration touches the shell, we get our desired statement.

The theorem holds, mutatis mutandis, for bond percolation models.
\begin{proof}
    We show that for each $\sigma \in \mathbb{Z}^d$, with high probability a constant proportion of the tiles in $A(r) \cap (R\mathbb{Z}^d+\sigma)$ are in $P_K$.  Since for some $\sigma$,
    \[|A(r) \cap (R\mathbb{Z}^d+\sigma)| \geq  |A(r)|/R^d,\]
    this is sufficient.
    
    Now we look at the disjoint $R$-cubes around each site $x \in A(r) \cap (R\mathbb{Z}^d+\sigma)$.  Write $Q_x=[0,R]^d+(x-x_0)$ and $K_x=K+(x-x_0)$.  Let $\rho_y$ denote the passage time of a site $y$, and let $X_x$ be the event that for all $y \in Q_x$,
    \begin{align*}
        \rho_y &\leq (R+2)^{-d} && \text{if }y \in K_x \\
        1<\rho_y &<2 && \text{if }y=x\\
        \rho_y &\geq 3 && \text{otherwise.}
    \end{align*}
    These times may be scaled based on the passage time distribution to make sure that the probability that $\rho_y$ lands in each range is nonzero.  Clearly all the $X_x$ are i.i.d.\ and each occurs with positive probability.  Therefore there is a constant $c(R)>0$ such that with high probability at least $c(R)|A(r) \cap (R\mathbb{Z}^d+\sigma)|$ of them occur.
    
    Now if $X_x$ occurs, and assuming $Q_x$ does not include the origin, let $r_x$ be the time at which $x$ enters the polyomino.  Then the path connecting the origin to $x$ has to go through the outermost layer of $Q_x$, so sites in that layer enter the polyomino earlier.  Once one point in $K_x$ is in the polyomino, the rest must join it after time $<1$.  The first point adjacent to $x$ joins at some time in $(r_x-2,r_x-1)$.  One sees therefore that all sites in $K_x$ must enter the polyomino at times in $(r_x-3,r_x)$, and that all sites in $Q_x \setminus (K_x \cup \{x\})$ must enter after $x$ does.  Thus every $x$ for which $X_x$ occurs is in $P_K$.
\end{proof}

Finally, we prove Theorem \ref{condChanges}, which states that under the assumption that  there is a $C(d)>0$ so that $P_d(t)\leq C(d) t^{(d-1)/d}$ with high probability, the probability of each specific change in $\beta_i$ occuring at time $t$ is asymptotically bounded away from zero, again with high probability.

By Theorem \ref{outermost} the perimeter contains $\geq c(d)t^{(d-1)/d}$ sites at time $t$ whose neighborhood looks like a given one of the configurations from Lemma \ref{jump_configs} with high probability.  Therefore, whenever $P_d(t) \leq C(d)t^{(d-1)/d}$, the probability that the next tile is added in the center of such a configuration is at least $c(d)/C(d)$.


\section{Computational experiments and open problems} \label{sec:comp}
Theorem \ref{thm_conditional} shows a rigorous asymptotic bound for the Betti numbers of the Eden growth model in $d$ dimensions. However, many finer questions about the associated geometry and topology remain open.  In this section, we investigate several of these questions via computational experiments for the Eden model in dimensions 2 through 5, giving evidence for Conjectures \ref{edenexp} and \ref{conj:bettiJumps} as stated in Section~\ref{S:main_results} and suggesting further conjectures.

The Eden Growth Model was implemented in Python, together with an algorithm that tracks the behavior of the $d-1$-dimensional homology at each timestep. We find a basis for $H_{d-1}\paren{A\paren{t}}$ via Alexander duality by identifying the bounded components of the complement and tracking how they change over time.  This implementation allows us to study fine questions about the distribution of shapes and area of the holes in the EGM in Section \ref{S:Compu_areasandshapes}. In Section \ref{S:Comp_per}, we also compute the proportion of the site perimeter contained in the unbounded component of the complement (the \emph{outer perimeter}) for clusters of sizes $1$ and $2$ million for the EGM in dimension two and clusters of size $1.5$ million for the EGM in dimensions 2 through 5. The data analyzed in Tables~\ref{table:perimeter}, \ref{table:areas}, and \ref{table:shapes}, and  Figures~\ref{fig:areas}, and \ref{F:max_hole} comes from a single set of 10 two-dimensional clusters of size $1$ million. This dataset has been made publicly available at the GitHub repository \cite{EdenSoftware}. 

The algorithm described in the previous paragraph cannot easily be modified to measure the local geometry associated with the lower-dimensional homology.\footnote{In this setting, the representative cycles of homology group elements are no longer unique, but intuitively one would like to choose and measure the ``smallest'' or ``tightest'' representative of each hole. This problem was recently studied for simplicial complexes in \cite{obayashi2018volume} and similar tools and techniques could be adapted and implemented for measuring the geometry associated to the homology of cubical complexes.} Instead, we use the Perseus software package~\cite{mischaikow2013morse,nanda2012perseus} to compute the Betti numbers and persistent homology in all dimensions. These computations are discussed in Sections \ref{S:Comp_betti} and \ref{S:Comp_persistent}. Unsurprisingly, this was slower than our other computations, and we include data from a single cluster of one million tiles for the two-dimensional EGM, and data from single clusters of size five hundred thousand for the EGM in dimensions three, four, and five.


\subsection{Total, inner, and outer perimeter} \label{S:Comp_per}

In applications of stochastic growth models (e.g.~to modeling a bacterial cell colony), the interaction with the medium takes place along the perimeter. These interactions may be qualitatively different for sites in the \emph{outer perimeter} (those that are contained in the unbounded component of the complement, where resources are unlimited) and sites in the \emph{inner perimeter} (those that are contained in the holes of the EGM). Top-dimensional holes can be thought of as \textit{capsules} whose contents cannot interact with the outside medium. In what follows, we analyze the total, inner and outer site perimeter of simulations of the Eden model in dimensions 2 through 5.

We remind the reader that the site perimeter of a $d$-dimensional polyomino $A$ is the set of $d$-cells that are not in $A$ but that have $(d-1)$-cells in common with $A$.

Table \ref{T:2dper} shows the mean and sample standard deviation of the sizes of the total, inner, and outer site perimeter for a sample of 10 simulations of the two-dimensional EGM, at sizes $10^5$ and $10^6$. From this data, we observe that the relative proportions of the inner and outer perimeters already appear to have stabilized by time $10^6$. This is further supported by data from two larger single cluster simulations of sizes $1.5$ million and $2$ million, for which $\operatorname{OutP}_2(t)/P_2(t)$ remains between $0.77$ and $0.80$ in all measurements taken once every $10^5$ timesteps. Thus, we conjecture that
\begin{equation} \label{OutP_bounds}
    0.77 \leq \frac{\operatorname{OutP}_2(t)}{P_2(t)} \leq 0.80
\end{equation}
with high probability as $t\rightarrow\infty$.

\begin{table}
\caption{Statistics of the total, inner, and outer site perimeters of a sample of 10 simulations of the two-dimensional Eden growth model up to times $10^5$ and $10^6$.}
\label{T:2dper}
    \centering
\begin{tabular} { c | c|c c || c c }
\multicolumn{2}{c|}{} & \multicolumn{2}{c||}{At time $10^5$:} & \multicolumn{2}{c}{At time $10^6$:}\\
\multicolumn{2}{c|}{} & Sample SD & Mean & Sample SD & Mean\\
\hline
\hline
\multicolumn{2}{c|}{Total (site) perimeter} & 29 & 2353 & 77    &  7594  \\
\hline
As fraction & Outer perimeter & 0.0114 & 0.7931 &  0.0068 &  0.7839  \\
\cline{2-6}
of total & Inner perimeter& 0.0114 &  0.2068 & 0.0068 & 0.2110 \\
\hline
\end{tabular}
\label{table:perimeter}
\end{table}

For the Eden model in dimensions 3 through 5, we performed the same computations for single clusters of size 1.5 million. Unlike the two-dimensional EGM, $\operatorname{OutP}_d(t)/P_d(t)$ was strictly decreasing for observations taken at evenly spaced intervals of $10^5$ timesteps. As such, we do not think that $\operatorname{OutP}_d(t)/P_d(t)$ has stabilized at time $t\leq 1.5 \times 10^6$ for $d=3$, $4$, or $5$ (see Table \ref{T:per_2-5D}). This is unsurprising given that the diameter of the clusters, in any sense, is proportional to $t^{1/d}$. Thus it is computationally infeasible to collect enough data to make reasonable conjectures about the limiting value of $\operatorname{OutP}_d(t)/P_d(t)$.  Nevertheless, we make the following conjecture:

\begin{table}
\caption{The total, inner, and outer perimeter and the diameter of one simulation in each dimension between 2 and 5 up to time 1.5 million.  The diameter given here is the sidelength of the smallest cube containing the polyomino; this behaves similarly to other possible notions of diameter.}
\label{T:per_2-5D}
    \centering
\begin{tabular} { c | c|c | c |c |c  }

\multicolumn{2}{c|}{} & 2D & 3D &4D & 5D\\
\hline
\hline
\multicolumn{2}{c|}{Total (site) perimeter} & 9,287 & 200,401 & 986,603 & 2,573,547 \\
\hline
As fraction & Outer perimeter & 0.7811 & 0.8150 & 0.9311 & 0.9950\\
\cline{2-6}
of total & Inner perimeter& 0.2188 & 0.1850 & 0.0689 & 0.005\\
\hline
\multicolumn{2}{c|}{Diameter} & 1424 & 165 & 64 & 40\\
\hline
\end{tabular}
\end{table}

\begin{conj} \label{C:Out}
    For each $d>1$, there is a number $\operatorname{per}_{d}>0$ so that
    \begin{equation} \label{OutPlimit}
        \frac{\operatorname{OutP}_d(t)}{P_d(t)} \rightarrow \operatorname{per}_{d}
    \end{equation}
    almost surely as $t\rightarrow\infty$.  Moreover,
    $\lim_{d \to \infty} \operatorname{per}_{d} = 1$.
\end{conj}

\subsection{Betti numbers}\label{S:Comp_betti}
In this section, we examine the asymptotics of the Betti numbers as well as the change in each Betti number at a single timestep. As mentioned before, the computations of the Betti numbers contained in this section were performed using the Perseus software package~\cite{mischaikow2013morse,nanda2012perseus}. Data for the two-dimensional Eden model comes from a single cluster of size one million, and data for dimensions three, four, and five are from single clusters of size 500,000.

\begin{figure}
    \centering
        \subfloat[]{\label{fig:changes_2D1}\includegraphics[width=0.45\textwidth]{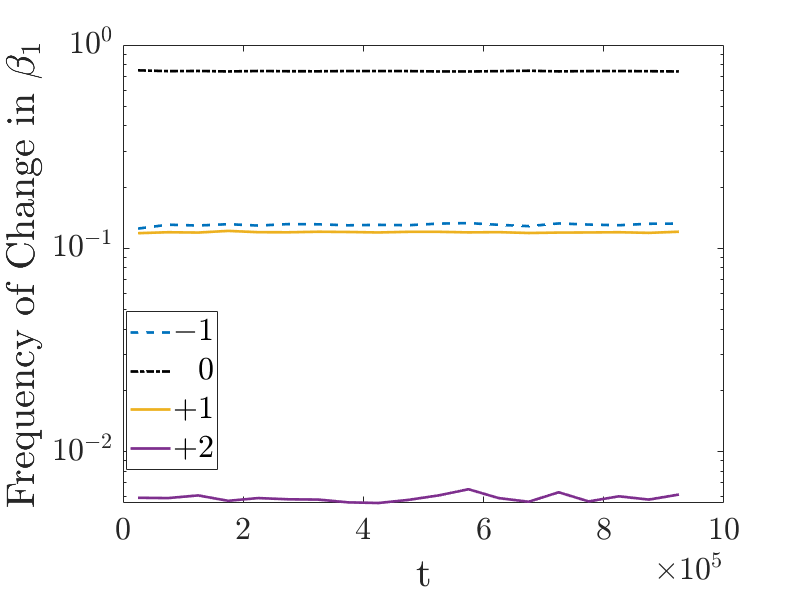}}
    \quad
        \subfloat[]{\label{fig:changes_3D1}\includegraphics[width=0.45\textwidth]{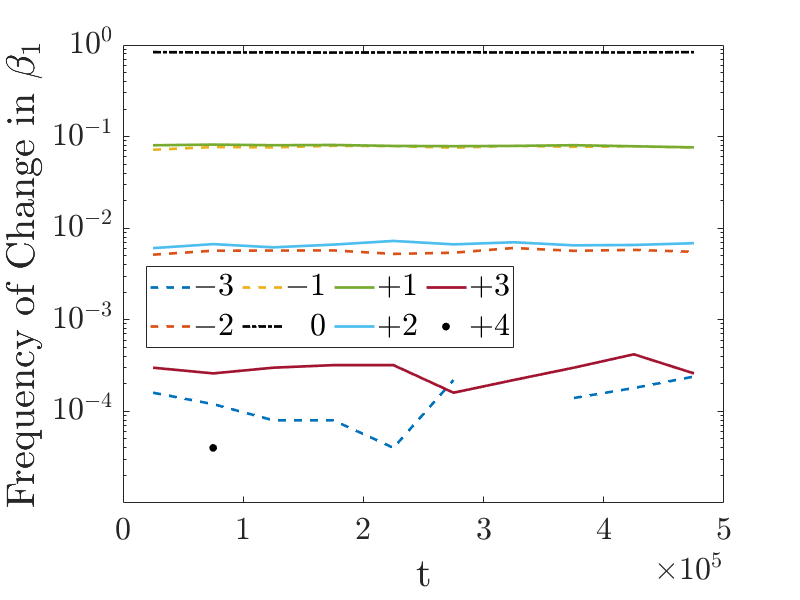}}
    
        \subfloat[]{\label{fig:changes_3D2}\includegraphics[width=0.45\textwidth]{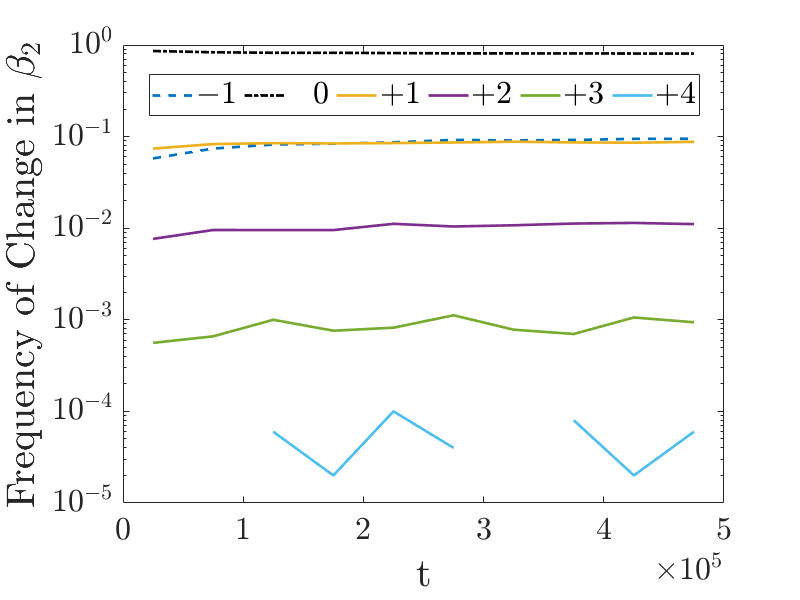}}
    \quad
        \subfloat[]{\label{fig:changes_4D1}\includegraphics[width=0.45\textwidth]{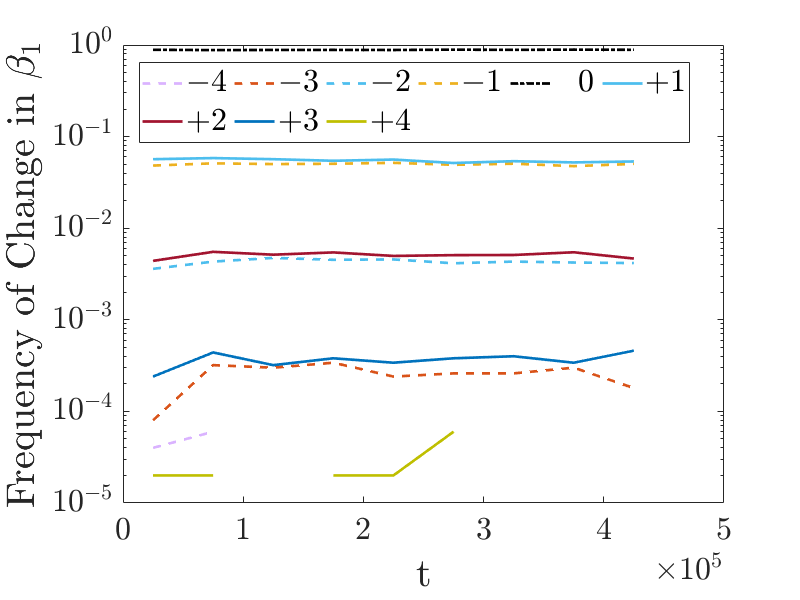}}
    
        \subfloat[]{\label{fig:changes_4D2}\includegraphics[width=0.45\textwidth]{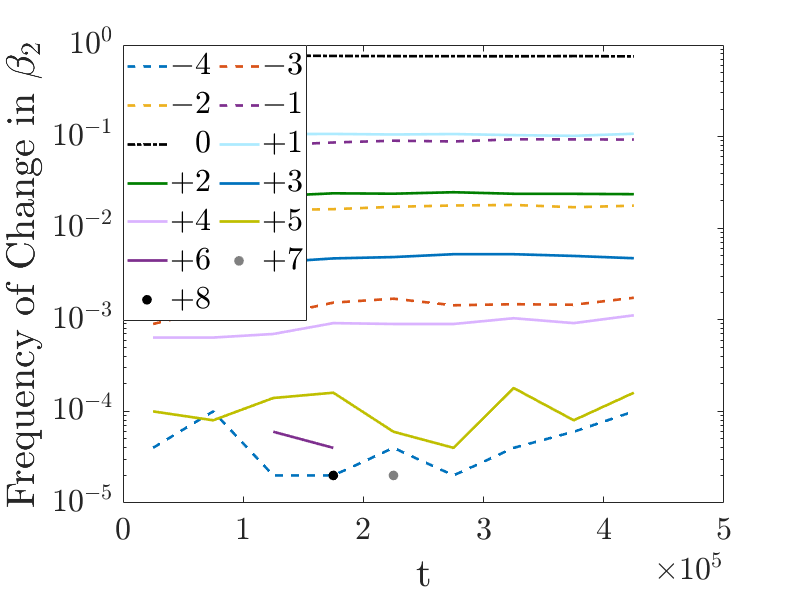}}
    \quad
        \subfloat[]{\label{fig:changes_4D3}\includegraphics[width=0.45\textwidth]{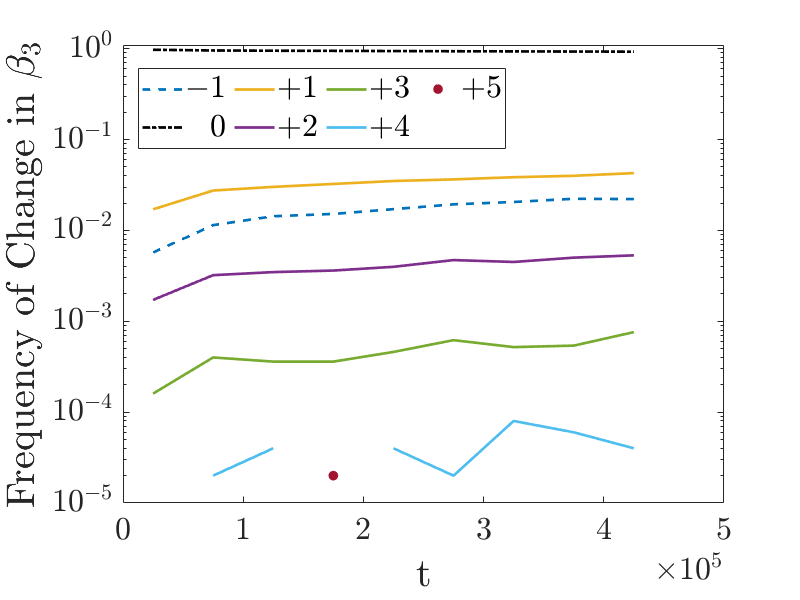}}
   
    \caption{Frequency with which each change in the Betti number in one timestep occurs in the (a) two-dimensional, (b)-(c) three-dimensional, and  (d)-(f) four-dimensional Eden model. Frequencies were averaged over bins of width 50,000 timesteps. This data provides strong evidence for Conjecture \ref{conj:bettiJumps}. Changes that occur in only one bin (e.g.~$\Delta\beta_1=+4$ in (b)) are shown by a dot instead of a line.
    }
    \label{fig:changeBetti}
\end{figure}

Figure \ref{fig:changeBetti} shows the frequencies of the event that a Betti number changes by a given amount in a single timestep. The frequencies of each event appear to converge quite quickly, providing strong evidence for Conjecture \ref{conj:bettiJumps}. Unsurprisingly, small jumps are much more frequent than large jumps. This is related to the closeness of the frequencies of $\beta_i$ increasing by one and decreasing by one in Figures~\ref{fig:changes_2D1}--\ref{fig:changes_4D2}: the total Betti number grows more slowly than the number of timesteps, so the number of positive changes balances out the number of negative changes, with an error term growing more slowly than $t$ (at a rate between $t^{d-1}/d$ and $P_d(t)$, by Theorem \ref{thm_conditional}). We expect this behavior to also occur for $\beta_3$ in the four-dimensional case, at larger values of $t$ than pictured in Figure \ref{fig:changes_4D3}. We provide more evidence below that statistics for this case have not yet stabilized. On the other hand, this heuristic does not explain the striking alignment in the frequency of events where $\beta_1$ changes by $+i$ and $-i$ in Figures~\ref{fig:changes_3D1} and~\ref{fig:changes_4D1}.

\begin{figure}
    \centering
        \subfloat[]{\label{fig:betti_2D}\includegraphics[width=0.48\textwidth]{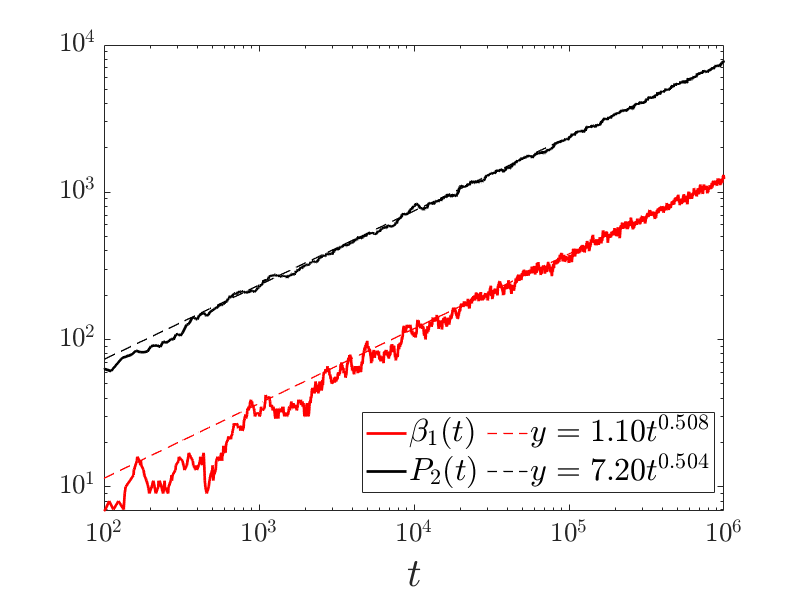}}
    \quad
        \subfloat[]{\label{fig:betti_3D}\includegraphics[width=0.48\textwidth]{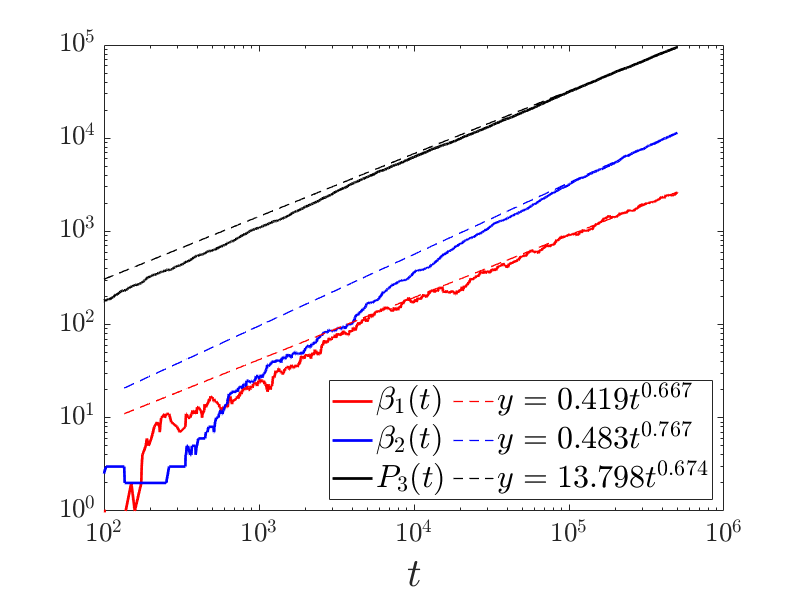}}
    
        \subfloat[]{\label{fig:betti_4D}\includegraphics[width=0.48\textwidth]{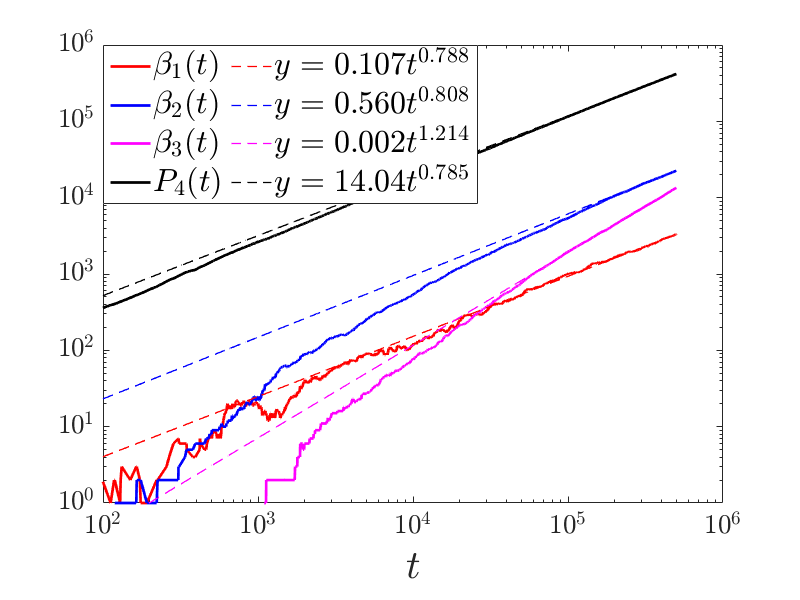}}
    \quad
        \subfloat[]{\label{fig:betti_5D}\includegraphics[width=0.48\textwidth]{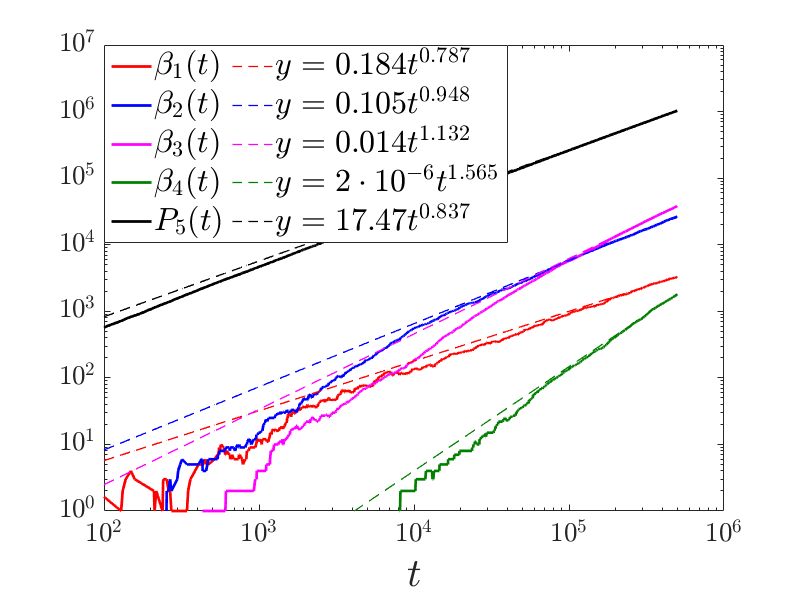}}

    \caption{The evolution of the Betti numbers and perimeter over time in the (a) two-dimensional, (b) three-dimensional, (c) four-dimensional, and (d) five-dimensional Eden models. Power laws were fitted in MATLAB.
    }
    \label{fig:bettiAsymptotics}

\end{figure}

The evolution of the Betti numbers over time are shown in Figure \ref{fig:bettiAsymptotics}, together with the perimeter. If $P_d(t)\sim t^{(d-1)/d}$ as conjectured, Theorem \ref{thm_conditional} would imply that $\beta_i(t)$ also scales as $t^{(d-1)/d}$. To test this, we fitted power laws to the Betti curves in MATLAB. Estimated exponents are relatively close to their conjectured values for $\beta_1$ for the Eden model in dimensions 2 through 5. Notably, $\beta_3$ in the four-dimensional Eden model and $\beta_3$ and $\beta_4$ in the five-dimensional Eden model are growing much faster than expected, at a rate exceeding that of the volume. We take this as further evidence that statistics have not stabilized in this case.

Another interesting trend in three and four dimensions is that the $\beta_i$ for small $i$ starts out larger at the beginning and is overtaken by $\beta_j$ for large $j$ as time goes on. Recall from Conjecture \ref{edenexp} that $C_{i,d}$ is the conjectured limit of $\beta_{i}(t)/t^{\frac{d-1}{d}}$  as $t\rightarrow\infty$. This data suggests a further conjecture.

\begin{conj}\label{conj:overtake}
For $0<i<j<d$, $C_{i,d}>C_{j,d}$.
\end{conj}
As we will see in the next section using persistent homology, a heuristic explanation for this behavior is that while higher dimensional homology classes form more infrequently than lower-dimensional ones, they last for much longer.

\subsection{Persistent Homology} \label{S:Comp_persistent}
When $\beta_i(t)$ changes, one would like to associate this with a specific geometric feature of $A(t)$ (an ``$i$-dimensional hole'') that forms or disappears at time $t$. In general, it is impossible to single out a specific such feature, as this requires a choice of basis for the $i$-dimensional homology and there are many reasonable choices (though the situation is clearer in codimension one, as we will see in the next section). However, there is a well-defined pairing between the events where an $i$-dimensional homology class is born and $\beta_i$ increases and the events where an $i$-dimensional homology class dies and $\beta_i$ decreases. This can be found using persistent homology.

Persistent homology~\cite{edelsbrunner2000topological} tracks the birth and death of homology generators over time. More precisely, if $X_1\hookrightarrow X_2 \hookrightarrow \cdots \hookrightarrow X_n$ is a filtration of topological spaces (that is, a sequence of topological spaces where each is a subset of the next), the $i$-dimensional persistent homology intervals $PH_i(\mathcal{X})$ are the unique set of half-open intervals $\left\{\left[\mathbf{b}_l,\mathbf{d}_l\right)\right\}$ with endpoints in $\left\{1,\ldots,n\right\}$ so that
\[\text{rank}\left(H_i\left(X_j\right)\rightarrow H_i\left(X_k\right)\right)=\#\left\{I\in PH_i(\mathcal{X}):\left[j,k\right]\subset I\right\}\,.\]
Compatible bases can be chosen for the homology groups $H_i\left(X_j\right)$ so that an interval $\left[\mathbf{b},\mathbf{d}\right)$ corresponds to a homology basis element that is born in $H_i(X_{\mathbf{b}})$, is mapped forward to basis elements in $H_i(X_j)$ for $\mathbf{b}<j<\mathbf{d}$, and dies in $H_i(X_{\mathbf{d}})$. Note that the choice of basis elements is not unique. For a more in depth introduction to persistent homology that describes further algebraic structure see, for example,~\cite{edelsbrunner2008persistent,chazal2016structure}.

Here, we compute the persistent homology of the natural filtration of the Eden growth model through time $A(1)\hookrightarrow A(2)\hookrightarrow\cdots \hookrightarrow A(t-1) \hookrightarrow A(t)$. This allows us to measure how long a homology class persists after it is born.

We first give a heuristic estimate for the expected persistence.  First, note that the persistence in first-passage perolation time of an element of $H_{d-1}(A(t))$ corresponding to a hole with one tile is exponentially distributed with mean $1.$ We claim the expectation scales as $t^{(d-1)/d}$ in Eden time.  To compute the expectation in Eden time, we need to estimate the expected difference in Eden time, i.e.\ volume, from $A(r)_{\mathrm{FPP}}$ to $A(r+s)_{\mathrm{FPP}}$, using this notation for the polyomino at FPP time $r$ and $r+s$, respectively.  For $s>>\sqrt{r}$, known convergence estimates for the shape theorem imply that this scales as $(r+s)^d-r^d$.  For smaller $s$, we use a heuristic.  We assume that as $u$ goes from $r$ to $2r$, $\mathbb{E}(|A(u+s)_{\mathrm{FPP}}|-|A(u)_{\mathrm{FPP}}|)$ changes at most by a multiplicative constant independent of $r$.  By splitting the interval $[r,2r]$ into smaller intervals and using the consequence of the shape theorem above,
\[\sum_{n=0}^{\lceil r/s \rceil} \mathbb{E}(|A(r+(n+1)s)_{\mathrm{FPP}}|-|A(r+ns)_{\mathrm{FPP}}|) \sim (2r)^d-r^d.\]
Dividing out and using our assumption, we get $\mathbb{E}(|A(r+s)_{\mathrm{FPP}}|-|A(r)_{\mathrm{FPP}}|) \sim sr^{d-1}$.  Integrating over $s$ with respect to the exponential distribution to get the expected Eden time, we see that the expected persistence of a hole with one tile scales as $t^{(d-1)/d}$, where $t \sim r^d$ is the Eden time, similar to the expected perimeter.  One might guess that the persistence of larger holes and holes of other dimensions follows a similar law; this is also suggested by our data.
\begin{figure}
    \centering
        \subfloat[]{\label{fig:PH_2D}\includegraphics[width=0.48\textwidth]{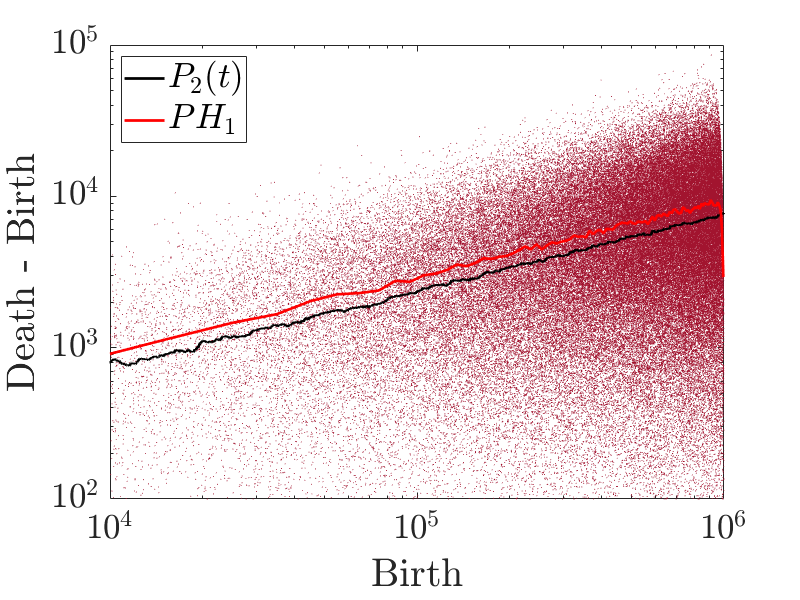}}
    \quad
        \subfloat[]{\label{fig:PH_3D}\includegraphics[width=0.48\textwidth]{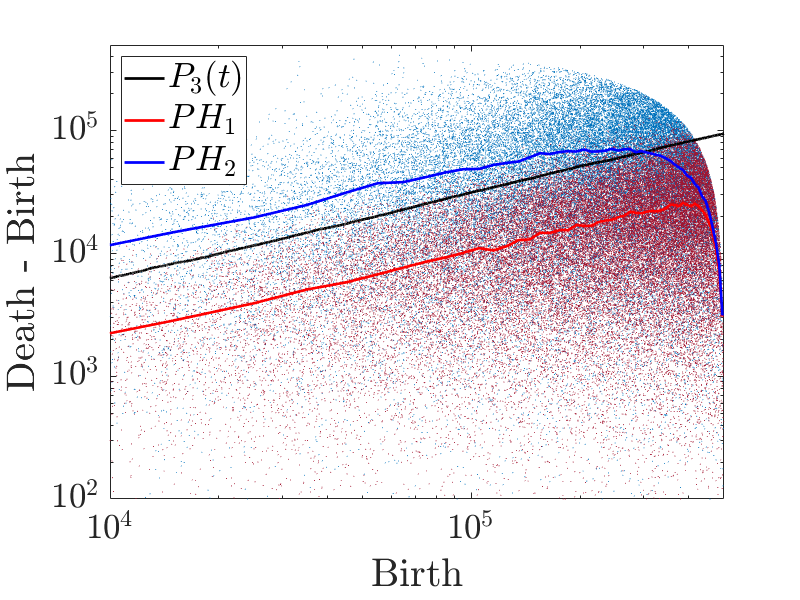}}
    
        \subfloat[]{\label{fig:PH_4D}\includegraphics[width=0.48\textwidth]{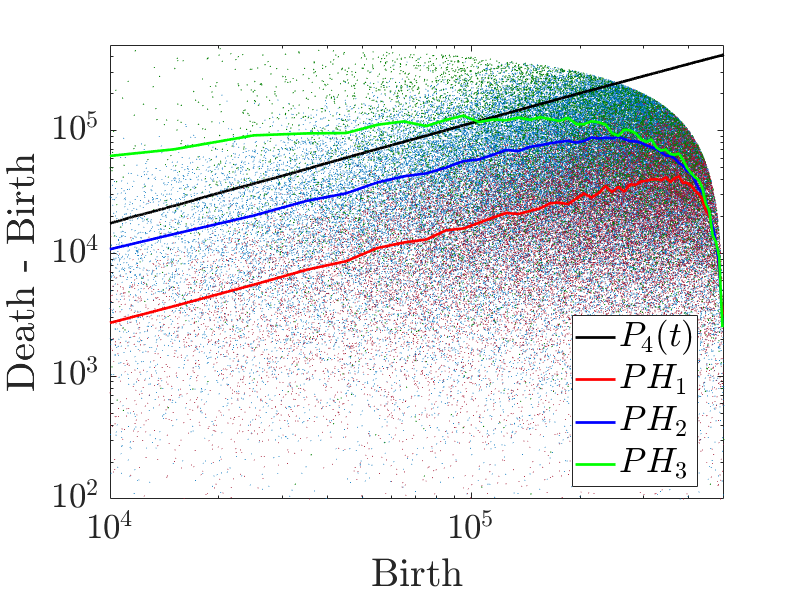}}
    \quad
        \subfloat[]{\label{fig:PH_5D}\includegraphics[width=0.48\textwidth]{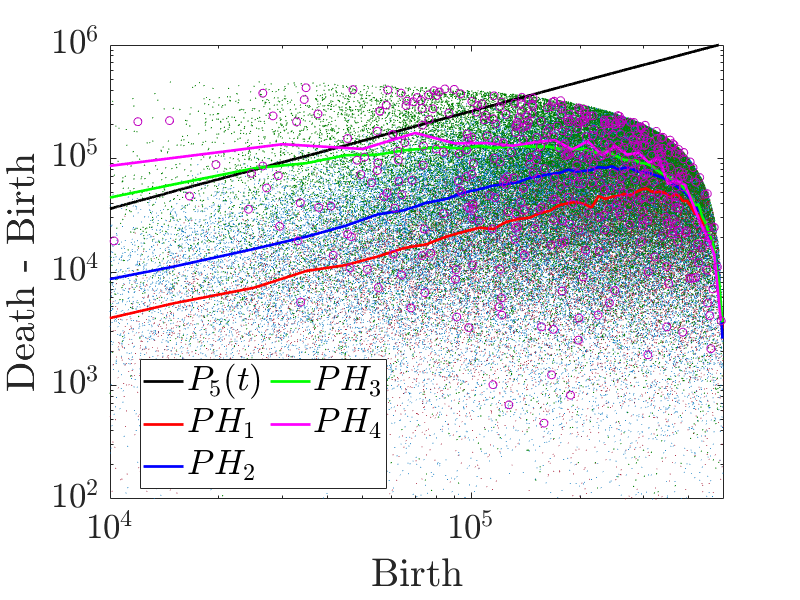}}

    \caption{Persistent homology scatter plots for the (a) two-dimensional, (b) three-dimensional, (c) four-dimensional and (d) five-dimensional Eden models The dropoff on the right of the figures is due to finite size effects. We used the Perseus software package~\cite{mischaikow2013morse,nanda2012perseus} to compute persistent homology.
    }
    \label{fig:PHPlots}
\end{figure}

The persistent homology data for the Eden model in dimensions 2--5 is shown in Figure \ref{fig:PHPlots}. While persistent homology is usually plotted in a scatter plot of birth versus death, we plot the birth versus the persistence to see how the expected persistence of a homology class changes over time. The scatter of points shows all intervals seen in the simulation, and the solid lines give an estimate of the average persistence of an interval with the given birth time. Note that the drop-off in the distribution of the deaths to the left of the plot is an artifact of the finite size of the simulation.  In all cases, higher-dimensional homology classes persist longer on average than lower-dimensional ones. This is unsurprising, as a homology class in $H_{d-1}$ can be killed only by adding specific tiles, but there are more ways to to kill lower-dimensional classes. On the other hand, there are more intervals for each dimension below $d-1$ (for example, for the four-dimensional EGM there are $3.4 \times 10^4$, $8.7\times 10^4$, and $2.2\times 10^4$ intervals in dimensions $1$, $2$, and $3$, respectively). These two trends explain the behavior we observed in Figure \ref{fig:bettiAsymptotics}, where the higher-dimensional Betti curves start below the lower-dimensional ones and then overtake them as time goes on: while fewer high-dimensional classes are born, they last much longer. Furthermore, most of the curves appear linear and parallel with $P_d(\mathbf{b})$ for a wide range, suggesting they follow a power law with the same exponent as $P_d(\mathbf{b})$.

\begin{figure}
    \centering
        \subfloat[]{\label{fig:PHist_3d}\includegraphics[width=0.48\textwidth]{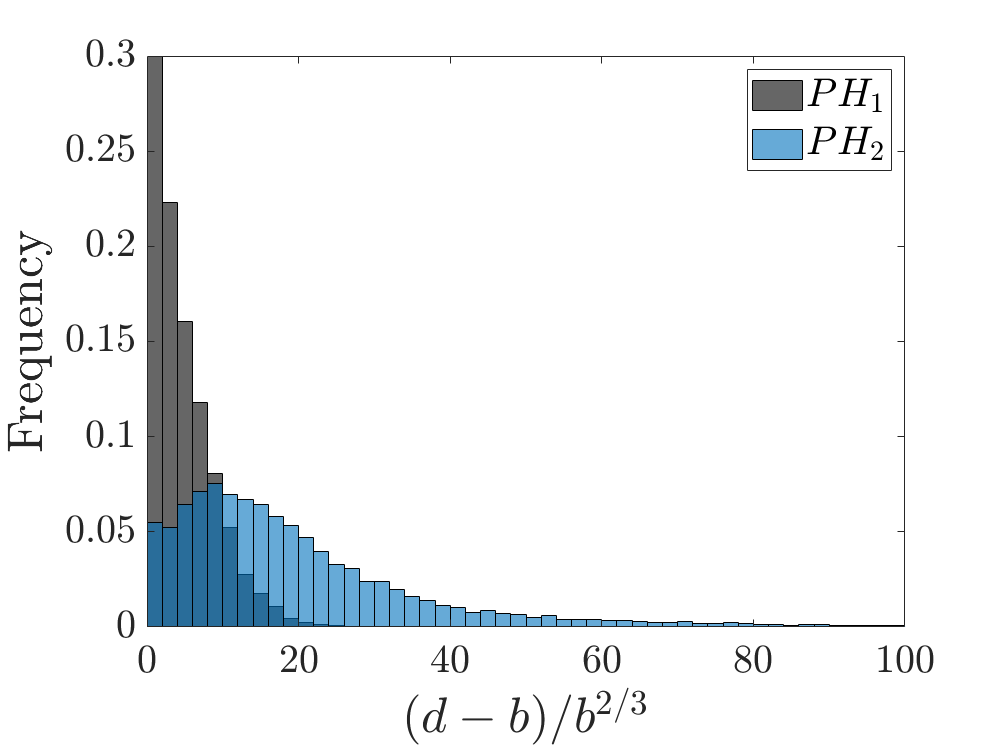}}
    \quad
        \subfloat[]{\label{fig:PHist_4d}\includegraphics[width=0.48\textwidth]{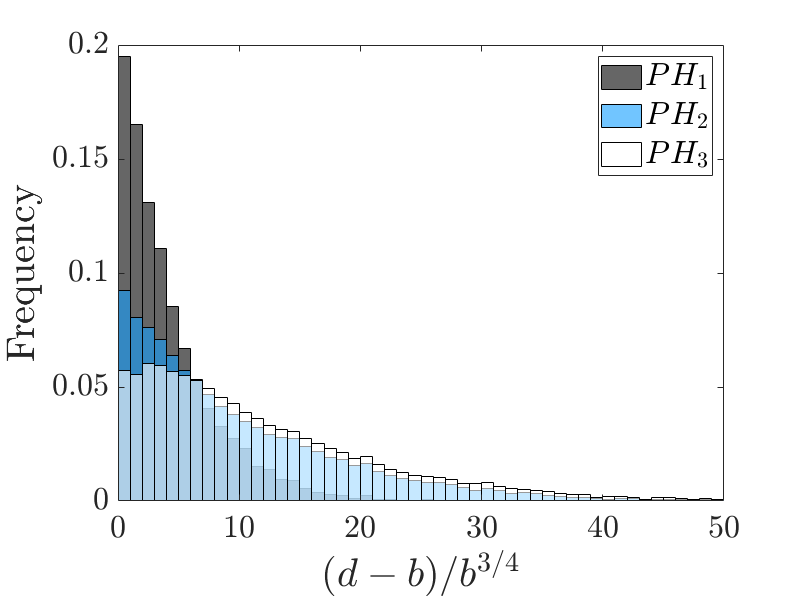}}
    \caption{Empirical distributions of the normalized persistence for the persistent homology of the Eden model in (a) three dimensions and (b) four dimensions. See the text for more details.
    }
    \label{fig:PHHistogram}

\end{figure}

The data in Figure \ref{fig:PHPlots} suggests that the expected persistence of an interval born at time $\mathbf{b}$ scales as the perimeter, which is believed to scale as $\mathbf{b}^{(d-1)/d}$. One might also suspect that the distribution of the normalized quantity $(\mathbf{d}-\mathbf{b})/\mathbf{b}^{(d-1)/d}$ converges as the birth time is taken to $\infty$. The empirical distribution of this normalized persistence is shown for Eden models in three and four dimensions in Figure \ref{fig:PHHistogram}. The figure includes data for intervals with birth times between $t=1 \times 10^5$ and $t=2\times 10^5$; the upper cutoff was chosen so only a small percentage of intervals born before that time persisted beyond $t=5\times 10^5$. (Computing the same histograms in a disjoint time interval results in a similar distribution.) Notably, the normalized persistence for $\PH_2$ has a substantially longer tail than that of $\PH_1$ for the three-dimensional Eden model, and both $\PH_3$ and $\PH_2$ have long tails for the four-dimensional model. For the latter case, it is somewhat surprising that the distribution for $\PH_2$ is more similar to that for $\PH_3$ than that for $\PH_1$.

\subsection{Local geometry of holes} \label{S:Compu_areasandshapes}
 
 In this section, we explore random variables defined in terms of the geometry associated to the $d-1$-dimensional homology of the EGM in $d$ dimensions.  These variables are: the areas, shapes, and evolution of the top-dimensional holes.

\subsubsection{Areas and shapes}
Betti numbers allow us to count the number of holes of each dimension. Alas, they tell us nothing about the geometry associated with these holes. As mentioned before, it is not easy to measure the geometric properties associated with homology in dimensions $1$ through $d-2$ as one cannot uniquely define representative cycles. Fortunately, for top-dimensional holes, we can use Alexander duality to associate generators of $H_{d-1}\paren{A\paren{t}}$ with components of the complement of $A\paren{t}$.  In what follows, we present statistics concerning the area and shapes of top-dimensional holes in simulations of the EGM, largely focusing on dimension 2.

\begin{figure}
    \centering
    \includegraphics[scale=.6]{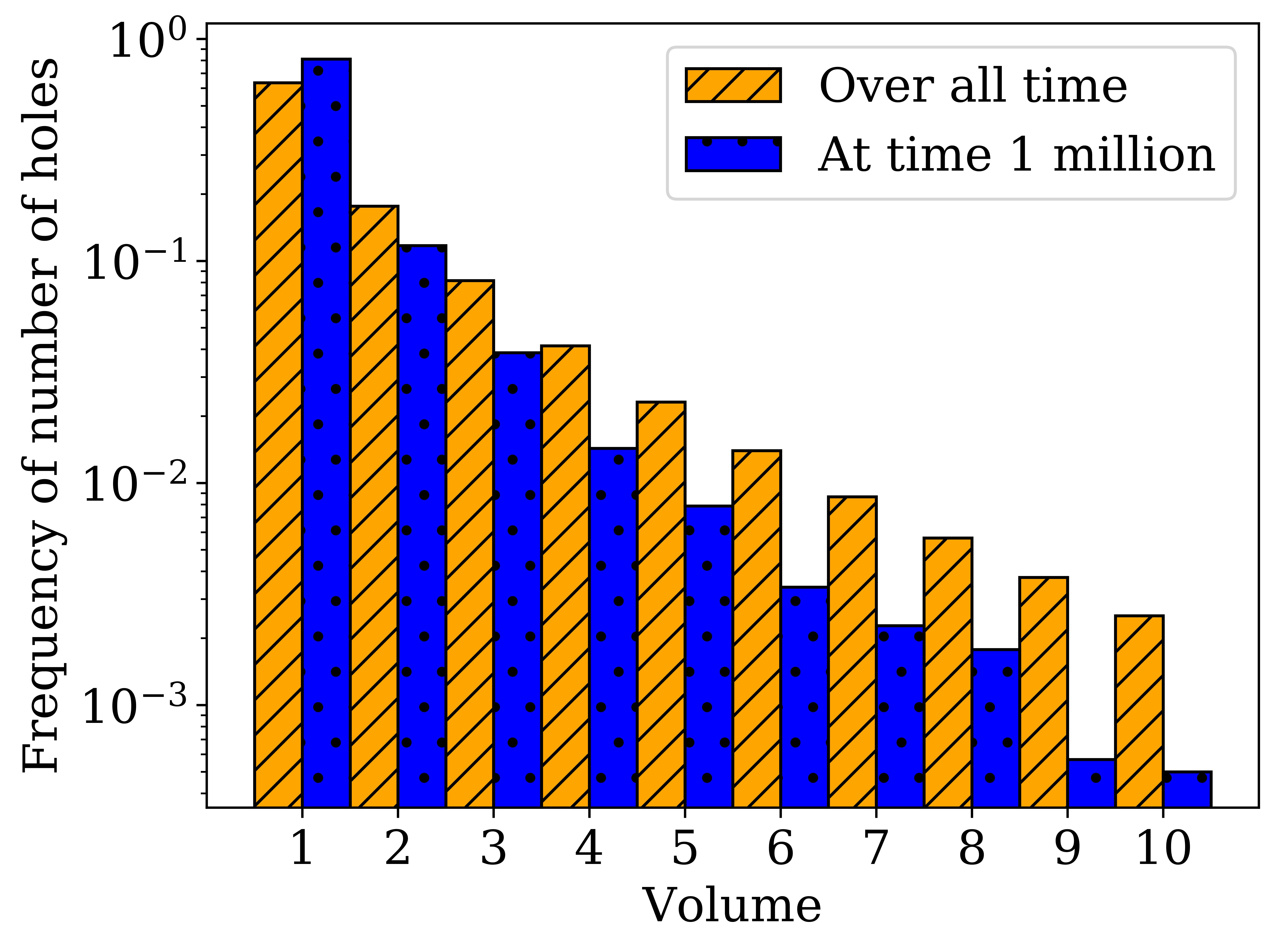}
    \caption{Histogram showing the (mean) frequency of the volumes of the holes at time 1 million and of all holes created at any time (measured at the birth of the corresponding persistent homology interval). The latter includes data both from holes created from the outer perimeter and those resulting from the division of an existing hole. Data was taken from 10 simulations of the 2D EGM up to time 1 million. Here we show only the frequency of holes up to area 10. More detailed statistical information is contained in Table \ref{table:areas}.}
    \label{fig:areas}
    \end{figure}
    \begin{table}
\caption{Table showing numerical data corresponding to Figure \ref{fig:areas}.}
\label{table:areas}
    \centering
\begin{tabular} {  c|c|c|c|c|c | c}
Area & 1 & 2 & 3 & 4 & 5 & $\geq 6$\\
 \hline
 \hline
 \multicolumn{7}{|c|}{At time $10^6$} \\
 \hline
\hline
 Mean & 0.812 & 0.117 &0.039 &0.014& 0.008& 0.010\\
 \hline
 SD & 0.011 & 0.008 & 0.006 & 0.005& 0.002& $ \leq 0.002$ \\
 \hline
 \hline
\multicolumn{7}{|c|}{Over all time} \\
\hline
\hline
 Mean & 0.636 & 0.177& 0.082& 0.042& 0.023& 0.042 \\
 \hline
SD &  \multicolumn{6}{c}{all values $ \leq 0.001$}\\
 \hline

\end{tabular}
\end{table}

Figure \ref{fig:areas} shows a histogram of the areas of holes in the two-dimensional EGM with respect to two distributions: the areas of the holes at the time they were born, taken over all time (in orange with diagonal lines), and the areas of the holes present at time 1 million (in blue with spots). Unsurprisingly, the areas of holes at the time they were born are slightly larger than the snapshot at time one million. The frequency of holes of a given area appears to decrease somewhat sub-exponentially as a function of area, although the relationship is less clear for the smaller sample. Table \ref{table:areas} shows the corresponding numerical data. Data was taken from 10 simulations of the two-dimensional EGM consisting of 1 million tiles.

Before studying the shapes of the holes in the two-dimensional EGM, we need to establish some conventions about how to count polyominoes of a given area. Two shapes are instances of the same \emph{fixed} polyomino if they are congruent after translation and of the same \emph{free} polyomino if they are congruent after rotations, reflections, and translations. For example, there are 19 fixed polyominoes of area four, and 5 free polyominoes of area four. All free polyominoes of areas three and four are depicted in Table \ref{table:shapes}, with the corresponding number of fixed polyominoes in the first row.

\begin{table}
    \caption{The proportion of holes of areas three and four which take the shape of each free polyomino. The statistical results summarized in this table were obtained from 10 different simulations of the EGM in two dimensions up to time one million.}
    \label{table:shapes}
    \centering
\begin{tabular} {  c|c|c|c|c|c || c|c }
 & 

\begin{tikzpicture}[scale=.45] 
\foreach \x/\y in { 0 / 0, 0 / 1, 1 / 0, 1 / 1  } { 
\path [draw=gray!66!black, fill=gray!66!black] (\x-0.45, + \y-0.45)
-- ++(0,.9)
-- ++(.9,0)
-- ++(0,-.9)
--cycle;
}
\end{tikzpicture} &

\begin{tikzpicture}[scale=.45] 
\foreach \x/\y in { 0 / 0, 0 / 1, 0 / 2, 1/1 } { 
\path [draw=gray!66!black, fill=gray!66!black] (\x-0.45,\y-0.45)
-- ++(0,.9)
-- ++(.9,0)
-- ++(0,-.9)
--cycle;
}
\end{tikzpicture}  &

\begin{tikzpicture}[scale=.45] 
\foreach \x/\y in {  0 / 2, 0 / 1, 1 / 0, 1/1 } { 
\path [draw=gray!66!black, fill=gray!66!black] (\x-0.45, \y-0.45)
-- ++(0,.9)
-- ++(.9,0)
-- ++(0,-.9)
--cycle;
}
\end{tikzpicture} &

\begin{tikzpicture}[scale=.45] 
\foreach \x/\y in {1 / 0, 1 / 1, 1 / 2, 0/2  } { 
\path [draw=gray!66!black, fill=gray!66!black] (\x-0.45, \y-0.45)
-- ++(0,.9)
-- ++(.9,0)
-- ++(0,-.9)
--cycle;
}
\end{tikzpicture}  &

\begin{tikzpicture}[scale=.45] 
\foreach \x/\y in {0 / 0, 0 / 1, 0 / 2,   0/3 } { 
\path [draw=gray!66!black, fill=gray!66!black] (\x-0.45,\y-0.45)
-- ++(0,.9)
-- ++(.9,0)
-- ++(0,-.9)
--cycle;
}
\end{tikzpicture}  & 

\begin{tikzpicture}[scale=.45] 
\foreach \x/\y in { 0 / 0, 0 / 1, 1 / 0  } { 
\path [draw=gray!66!black, fill=gray!66!black] (\x-0.45, + \y-0.45)
-- ++(0,.9)
-- ++(.9,0)
-- ++(0,-.9)
--cycle;
}
\end{tikzpicture} &

\begin{tikzpicture}[scale=.45] 
\foreach \x/\y in { 0 / 0, 0 / 1, 0 / 2 } { 
\path [draw=gray!66!black, fill=gray!66!black] (\x-0.45, + \y-0.45)
-- ++(0,.9)
-- ++(.9,0)
-- ++(0,-.9)
--cycle;
}
\end{tikzpicture}

\\
Number of fixed types (R) & 1 & 4 & 4 & 8 & 2 & 4 & 2 \\
 \hline
 \hline
 
 \multicolumn{8}{|c|}{Over all time:} \\
 \hline
Mean frequency (F) & 0.131 & 0.248 & 0.227 & 0.338 & 0.055 &  0.736 & 0.263   \\
 \hline
F /R  & 0.131 & 0.062  & 0.056 &  0.042  &  0.027 &  0.184 & 0.131 \\
\hline
Sample SD & 0.004 & 0.004 & 0.007 & 0.007 & 0.003 &0.003 & 0.003 \\
\hline
 \hline
 
 \multicolumn{8}{|c|}{At time 1 million:} \\
 \hline
Mean frequency (F) & 0.066 & 0.278 & 0.249 & 0.341 & 0.064 & 0.763 & 0.236 \\
 \hline
F /R  & 0.066 & 0.069  & 0.062 &  0.042  &  0.032 &  0.190 & 0.118 \\
\hline
Sample SD & 0.072 & 0.079 & 0.106 & 0.078 & 0.031 & 0.082 & 0.082\\
\hline

\end{tabular}
\end{table}

In Table \ref{table:shapes}, we show the proportion of holes in the two-dimensional EGM that take the shape of each free polyomino of area three or four. The data was taken from 10 different runs of the Eden model through time 1 million. In this sample, we observed an average of 147,306.5 holes of all sizes with a sample standard deviation of 152.4, of which an average of 6,113.5 holes had area four and 12,026.6 had area three at the moment of their birth, with sample standard deviations of 97.1 and 54.8 respectively. At time 1 million, we observed an average of 1,231.5 holes of all sizes with a sample standard deviation of 34.2, of which an average of 17.7 holes had area four and 47.6 holes had area three, with sample standard deviations of 6.2 and 7.4 respectively (in Table \ref{table:areas} these statistics are presented as frequencies). Note that the most common birth shape of area four is the roundest (the square) when controlling for multiplicity, and the least common is the longest. However, at time 1 million, the T-shape just edges out the square. The difference between these frequencies is likely related to properties of the ``reverse process'' we describe in the next section.

We have also recorded the extremal volumes of holes in dimension 2 through 5. In ten simulations of the two dimensional Eden model up to time 1 million, the largest hole created had an area of 48.7 and a standard deviation of 10.564. One of these largest holes is depicted in Figure \ref{F:max_hole}, together with the largest hole created in a simulation of the 3D EGM up to time 1 million. In Table \ref{table:biggest}, we record the volume of the largest top-dimensional hole created in a single simulation through time 1.5 million in each dimension from 2 to 5. 

\begin{figure}
\centering
\subfloat{
\centering
    \begin{tikzpicture}[scale=.5] 
\foreach \x/\y in {4 / 10, 3 / 10, 4 / 9, 3 / 9, 5 / 9, 4 / 8, 2 / 9, 3 / 8, 6 / 9, 4 / 7, 1 / 9, 2 / 8, 3 / 7, 5 / 7, 4 / 6, 1 / 8, 2 / 7, 3 / 6, 5 / 6, 4 / 5, 1 / 7, 3 / 5, 6 / 6, 5 / 5, 4 / 4, 3 / 4, 6 / 5, 5 / 4, 4 / 3, 3 / 3, 7 / 5, 6 / 4, 5 / 3, 4 / 2, 2 / 3, 8 / 5, 7 / 4, 4 / 1, 1 / 3, 2 / 2, 9 / 5, 5 / 1, 0 / 3, 1 / 4, 1 / 2, 0 / 4, 1 / 1, 1 / 0} { 
\path [draw=gray!66!black, fill=gray!66!black] ( 8.5+\x-0.45, 3.5+\y-0.45)
-- ++(0,.9)
-- ++(.9,0)
-- ++(0,-.9)
--cycle;
} \end{tikzpicture}
}
\qquad
\subfloat{
\centering
    \includegraphics[width=5cm]{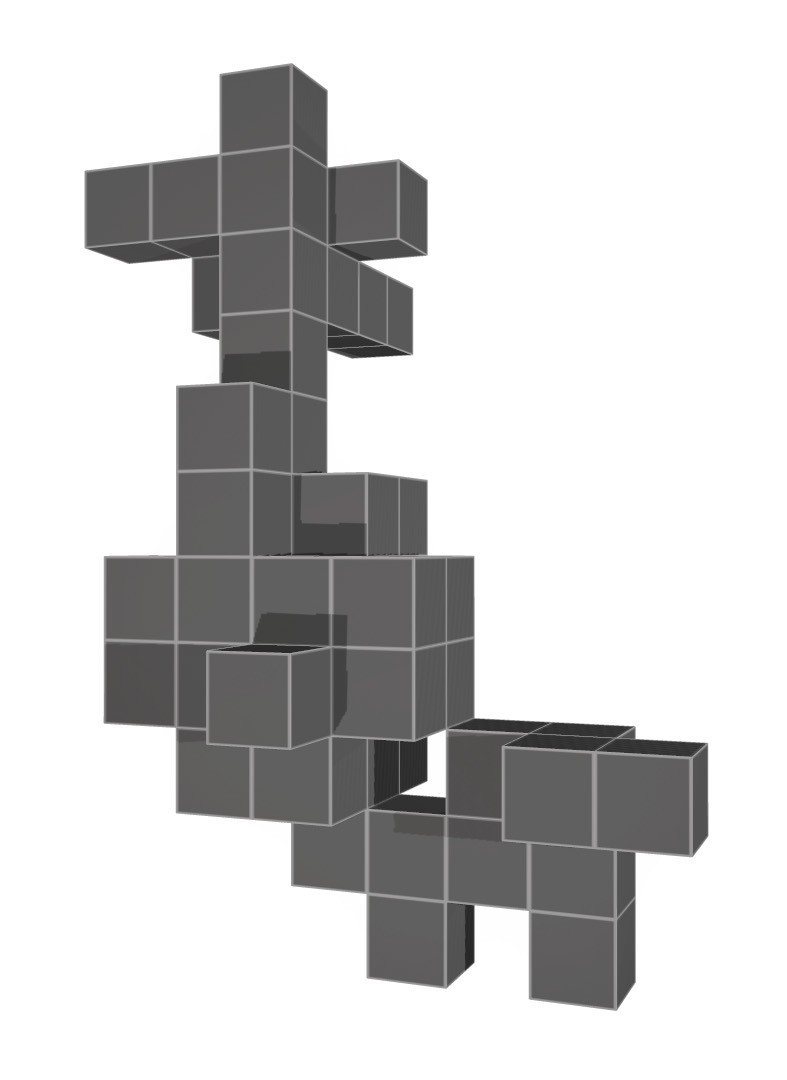}
}
\caption{A ``cast'' of the largest top-dimensional holes. The polyomino on the left has area 48 and the polycube on the right has volume 49. They were obtained from a simulation of the EGM up to time one million in two and three dimensions respectively.  In the two-dimensional case, this simulation was chosen randomly out of 10 runs of the experiment. The polycube is available for interactive exploration at \url{https://skfb.ly/6SnzN}.}
\label{F:max_hole}
\end{figure}
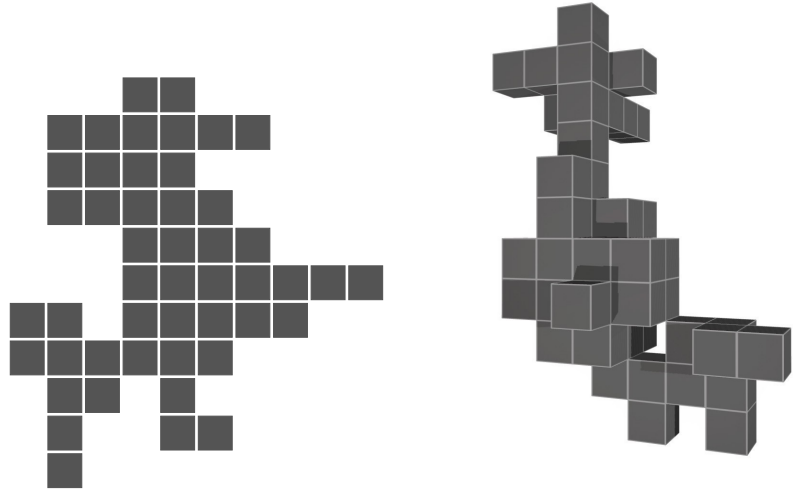

\begin{table}
\caption{Volumes of the largest top-dimensional holes created in an EGM simulation up to time 1.5 million in each of dimensions 2 through 5.}
\label{table:biggest}
    \centering
\begin{tabular} {  c|c|c|c|c}

 & 2D & 3D & 4D & 5D  \\
 \hline
\hline
Largest volume at time 1.5 million  &10 &27&32&19\\
 \hline
Largest volume over all time & 48 & 64 & 55 & 30 \\
 \hline

\end{tabular}
\end{table}
\subsubsection{Splitting trees}
After a hole forms from the outer perimeter, it may split a number of times before disappearing. This behavior is captured by a splitting tree~\cite{schweinhart2015statistical}, which tracks the times that division occurs and the resulting polyominoes. Note that these splitting times correspond to births of intervals in the $d-1$-dimensional persistent homology. In Figure \ref{F:spliiting_tree}, we show the splitting tree of the two-dimensional hole depicted in Figure \ref{F:max_hole}. We do not perform an in depth analysis of this data, but we propose that the ``reverse process'' that produces this splitting tree is of interest. More precisely, $P(t)$ evolves by the reverse process with initial condition $P(0)$ if $P(t)$ is determined from $P(t-1)$ by uniformly removing one of the tiles adjacent to the perimeter.   This is equivalent to applying the Eden growth process to the complement of $P(0)$.

\begin{figure}
\begin{center}
\includegraphics[width=\textwidth]{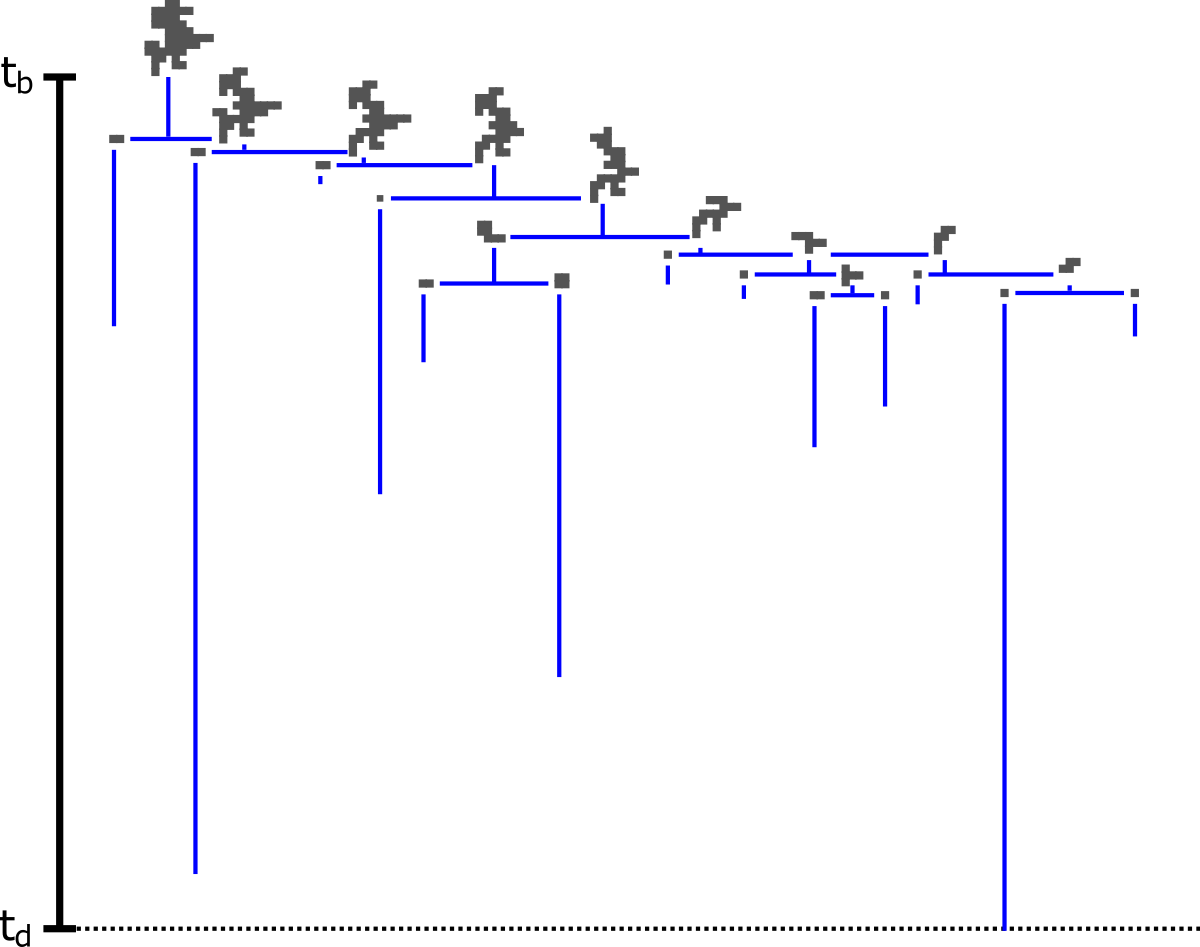}
\end{center}
\caption{Splitting tree of the two-dimensional hole depicted in Figure \ref{F:max_hole}. Its birth time is $t_b = 586,942$ and its death time is $t_d = 618,185$.}
\label{F:spliiting_tree}
\end{figure}
\section*{Acknowledgments}
We thank Ulrich Bauer, Eric Babson, Michael Damron, Christopher Hoffman, Matthew Kahle, Sayan Mukherjee, and Kavita Ramanan for interesting conversations about the Eden model.  We are also grateful for helpful comments from Anna Krymova and two anonymous referees which have greatly improved the paper.
\bibliographystyle{plain}
\bibliography{example}

\end{document}